\documentclass{amsart}
\usepackage[francais]{babel}
\usepackage[latin1]{inputenc}
\usepackage{amssymb, amsmath}
\usepackage[hypertex]{hyperref}

\numberwithin{equation}{section}

\theoremstyle{plain}
\newtheorem{theorem}{Th\'eorème}[section]
\newtheorem{corollary}[theorem]{Corollaire}
\newtheorem{prop}[theorem]{Proposition}
\newtheorem{lemma}[theorem]{Lemme}
\newtheorem*{thmA*}{Th\'eorème A}
\newtheorem*{thmB*}{Th\'eorème B}
\newtheorem*{thmC*}{Th\'eorème C}

\theoremstyle{definition}
\newtheorem*{question}{Question}
\newtheorem{remark}[theorem]{Remarque}
\newtheorem{fact}[theorem]{Fait}
\newtheorem{definition}[theorem]{D\'efinition}

\newcommand{\Q}{\mathbb{Q}}

\newcommand{\cb}{\operatorname{Cb}}

\newcommand{\St}{\operatorname{Stab}}

\newcommand{\acl}{\operatorname{acl}}
\newcommand{\aclq}{\operatorname{acl}^\mathrm{eq}}
\newcommand{\tr}{\operatorname{degtr}}
\newcommand{\ld}{\operatorname{dimlin}}

\newcommand{\sscl}[1]{\langle #1 \rangle}
\newcommand{\inv}[1]{ #1^{-1}}

\newcommand{\dd}{\delta}

\newcommand{\tp}{\operat\begin{normalsize}

\end{normalsize}orname{tp}}
\newcommand{\RM}{\operatorname{RM}}

\def\bdd{\mathrm{bdd}}
\def\tp{\mathrm{tp}}
\def\stp{\mathrm{stp}}
\def\U{\ddot{U}}

\def\a{\overline{a}}
\def\ab{\overline{ab}}
\def\ca{\overline{ca}}
\def\cab{\overline{cab}}
\def\b{\overline{b}}
\def\cc{\overline{c}}
\def\coker{\mathrm{coker}}

\def\sse{\ast}

\def\Ind#1#2{#1\setbox0=\hbox{$#1x$}\kern\wd0\hbox to 0pt{\hss$#1\mid$\hss}
\lower.9\ht0\hbox to 0pt{\hss$#1\smile$\hss}\kern\wd0}
\def\Notind#1#2{#1\setbox0=\hbox{$#1x$}\kern\wd0\hbox to
0pt{\mathchardef\nn="0236\hss$#1\nn$\kern1.4\wd0\hss}\hbox to
0pt{\hss$#1\mid$\hss}\lower.9\ht0
\hbox to 0pt{\hss$#1\smile$\hss}\kern\wd0}
\def\ind{\mathop{\mathpalette\Ind{}}}

\def\indi#1{\mathop{\ \ \hbox to 0pt{\hss$\mid^{\hbox to
0pt{$\scriptstyle#1$\hss}}$\hss}
\lower4pt\hbox to 0pt{\hss$\smile$\hss}\ \ }}
\def\nindi#1{\mathop{\ \ \hbox to 0pt{\hss$\!\not{\mid}^{\hbox to
0pt{$\scriptstyle\,#1$\hss}}$\hss}
\lower4pt\hbox to 0pt{\hss$\smile$\hss}\ \ }}

\begin{document}

\title{\`A la recherche du tore perdu}
\date{\today}

\author{Thomas Blossier, Amador Martin-Pizarro et Frank O. Wagner}
\address{Universit\'e de Lyon; CNRS; Universit\'e Lyon 1; Institut 
Camille Jordan UMR5208, 43 boulevard du 11
novembre 1918, F--69622 Villeurbanne Cedex, France }
\email{blossier@math.univ-lyon1.fr}
\email{pizarro@math.univ-lyon1.fr}
\email{wagner@math.univ-lyon1.fr}
\thanks{Recherche conduite dans le cadre du projet ANR-09-BLAN-0047 MODIG ainsi
que ANR-13-BS01-0006 ValCoMo.}

\begin{abstract} Un groupe interprétable dans le mauvais corps vert
est isogène à un quotient d'un sous-groupe définissable d'un groupe algébrique par 
une puissance du groupe vert. Un sous-groupe définissable d'un groupe
algébrique dans un corps vert ou rouge est une extension des points 
colorés d'un groupe algébrique multiplicatif ou additif
par un groupe algébrique. En particulier, tout groupe simple
définissable dans un corps coloré est algébrique.
\end{abstract}
\maketitle

\section*{Introduction}
La conjecture de l'algébricité des groupes simples infinis
$\aleph_1$-catégoriques énoncée par Cherlin et Zilber affirme qu'un
tel groupe s'interprète comme un
groupe al\-gé\-bri\-que sur un corps algébriquement clos. Cette conjecture
a donné  naissance à une interaction entre la théorie des modèles et
la théorie des groupes.

Un {\em mauvais groupe} serait un contre-exemple de
dimension minimale à la conjecture. La dimension en question est le
{\em rang de Morley}, c'est-à-dire le rang de Cantor-Bendixon de l'espace des
types sur un modèle suffisamment saturé. Un des premiers obstacles
pour la caractérisation algébrique des sous-groupes de Borel
(sous-groupes définissables connexes résolubles maximaux) d'un groupe infini simple
$\aleph_1$-catégorique est l'existence éventuelle d'un {\em mauvais
corps}~: un corps infini de rang de Morley fini muni d'un prédicat pour un
sous-groupe multiplicatif divisible non trivial. Rappelons qu'un corps
infini de rang de Morley fini est algébriquement clos et que, dans le
langage pur des anneaux, le rang de Morley équivaut à la dimension de
Zariski.

L'existence d'un mauvais corps en caractéristique positive est improbable
\cite{Wa03}. En caractéristique nulle, Poizat \cite{Po01} a utilisé la
construction par amalgamation déve\-lop\-pée par Hrushovski \cite{Hr92, Hr93}
pour introduire un corps de rang $\omega\cdot2$ avec un prédicat pour un
sous-groupe multiplicatif de rang $\omega$.
Ce corps a été ensuite collapsé pour obtenir un mauvais corps
\cite{BHPW06}. Dans le même article, Poizat a également construit un corps de
caractéristique positive avec un prédicat pour un sous-groupe additif
infini d'indice infini, collapsé dans \cite{BMPZ05}. Ce travail
faisait suite à sa construction \cite{Po99}  d'un corps avec un
prédicat pour un
sous-ensemble algébriquement indépendant, collapsé dans \cite{BH00}.
Poizat nomme les prédicats
\emph{vert}, \emph{rouge} et \emph{noir} respectivement. Les corps ainsi
obtenus (collapsés ou non) s'appellent donc les corps {\em colorés}.

L'étude systématique des groupes définissables dans les corps colorés
a été entamée dans \cite{BPW09}. Il découle du \cite[Théorème
5.7]{BPW09} que tout groupe définissable
connexe $G$ dans un corps coloré s'envoie par un morphisme
définissable $\phi$  dans un groupe algébrique, tel que la composante
connexe du noyau est incluse dans le centre de $G$. Tout 
groupe simple définissable est ainsi linéaire. Par \cite{Polong},
aucun mauvais groupe n'est définissable dans un corps rouge collapsé
et un mauvais groupe définissable dans un corps vert collapsé
serait un groupe linéaire constitué uniquement d'éléments
semi-simples.

Puisque les corps de rang de Morley fini éliminent
les imaginaires, on peut remplacer {\em définissable} par {\em
interprétable} pour les corps collapsés. Néanmoins, le résultat
précédent ne donne aucune information sur les
groupes abéliens.  En particulier, on ne peut traiter ainsi le
quotient du groupe multiplicatif par le sous groupe coloré. 

Dans la troisième partie de cet article (théorème \ref{T:theoremeA}),
nous donnons une caractérisation des groupes
définissables dans les corps verts collapsés.

\begin{thmA*}
Un groupe interprétable dans un corps vert collapsé est isogène à un
quotient d'un sous-groupe définissable d'un groupe algébrique par un
sous-groupe central qui est lui isogène à une puissance cartésienne du
sous-groupe multiplicatif vert.
\end{thmA*}

On utilise à plusieurs reprises le fait que les sous-groupes
multiplicatifs définis\-sables connexes sont des tores,
qui sont définissables sans paramètres. La preuve de ce
théorème ne s'applique donc pas directement au cas
rouge collapsé, où l'on doit traiter les sous-groupes additifs donnés
par des $p$-polynômes.

Le théorème A nous amène naturellement à étudier, dans les corps
colorés, les sous-groupes définissables d'un groupe algébrique.
Dans le cas du corps noir non-collapsé, Poizat montre que tout
groupe définissable est algébrique
\cite[Proposition 2.4]{Po99}. Ceci reste vrai (voir la conclusion du
même article page 1354) pour les sous-groupes définissables
de groupes algébriques dans le cas collapsé.
Dans la quatrième partie, qui est indépendante du théorème A, nous 
caractérisons les sous-groupes d'un groupe algébrique, définissables
dans un corps vert (théorème \ref{T:Bvert}) ou rouge (théorème
\ref{T:Brouge}), collapsé ou non.

\begin{thmB*}
Dans un corps vert (resp.\ rouge), tout sous-groupe
définissable connexe $G$ d'un groupe algébrique a un sous-groupe normal
algébrique $N$, tel que le quotient $G/N$ est définissablement isomorphe 
(resp.\ définissablement isogène) à une puissance cartésienne du sous-groupe
multiplicatif vert (resp.\ au sous-groupe rouge d'un groupe algébrique
additif).
\end{thmB*}

Ainsi, tout groupe simple définissable dans un corps
coloré est algébrique (corollaire \ref{C:Bsimple}). En particulier,
par élimination des imaginaires, le mauvais corps construit dans
\cite{BHPW06} n'interprète pas de mauvais groupe.

Dans la dernière partie, on étudie les groupes définissables dans une
fusion de deux théories fortement minimales avec la propriété de la
multiplicité définissable (DMP) par amalgame de Hrushovski
\cite{Hr92}. Cette construction a été généralisée \cite{mZ08} au cas
de structures de rang fini avec la DMP. Notons qu'un
corps coloré peut aussi être considéré comme une telle fusion. 
Hrushovski a construit ces fusions pour deux théories fortement minimales  
ayant comme réduit commun l'égalité (cas du corps noir)~; dans le cas
où le réduit commun est un espace vectoriel sur un corps fini (cas du
corps rouge), cette construction a été réalisée dans \cite{BMPZ06}.
Dans le cas où le réduit commun est un espace vectoriel sur $\Q$ (celui du
corps vert), une telle fusion n'a pas encore été construite en
général.

Hrushovski a affirmé que tout groupe définissable dans la fusion
fortement minimale de deux théories fortement minimales sur l'égalité
est isogène à un produit direct de deux groupes, chacun définissable
dans l'une des théories initiales. Il n'en a pas donné de preuve, et
seulement une esquisse se trouve dans \cite[Claim 3.1]{Ha08}. En nous
inspirant de la preuve qu'aucun groupe infini n'est définissable dans
une théorie plate \cite[Section 4.2]{Hr92}, nous donnons (théorème
\ref{T:fusion}) une démonstration complète de ce résultat, valable pour la
fusion libre ou collapsée.

\begin{thmC*}
Dans une fusion sur l'égalité de deux théories fortement minimales
avec la DMP, tout groupe définissable est, modulo un noyau
fini, isomorphe à un produit direct de groupes interprétables dans les
théories de base.
\end{thmC*}

\section{Préliminaires}
Dans cette partie, nous rappelons des résultats portant sur les
groupes stables qui seront utilisés par la suite. Nous renvoyons le lecteur à
\cite{Po87} ou \cite[Chapter 1 Part 6]{Pi96} pour une
introduction aux groupes stables.
\begin{definition} Soit $G$ un groupe stable. \'Etant donnés un
ensemble $A$ de paramètres et un élément $g$ de $G$, le
\emph{stabilisateur} (à gauche) de $g$ dans $G$ sur $A$ est le
sous-groupe défini par
$$\St(g/A) = \{h\in
  G\ :\ \exists\,x\models\stp(g/A)\ (\,hx\models\stp(g/A) \ \land
  x\ind_A h\,)\}.$$
\end{definition}
Notons que $\St(g/A)=\St(g/B)$ si $g\ind_A  B$. De plus, si
$h\in\St(g/A)$
et $g'\models\stp(g/A)$ avec $g'\ind_Ah$, par stationnarité on conclut
que $hg'\models\stp(g/A)$ et $hg'\ind_Ah$.

Le lemme suivant est dû à Ziegler dans le cas abélien 
\cite[Theorem~1]{mZ90}.

\begin{lemma}\label{L:stab} Soit $H$ un groupe type-définissable dans
une théorie stable et $h$, $h'$ dans $H$ tels que $h$, $h'$ et $hh'$ sont
deux-à-deux indépendants sur un ensemble $A$.
Alors les stabilisateurs à gauche satisfont
$$\St(h/A)=\St(hh'/A)=h\St(h')h^{-1}\,;$$
ils sont connexes, et $h$ est générique dans le translaté $\St(h/A)h$, qui
est définissable sur $\aclq(A)$ ; de même pour $h'$ et $hh'$.
\end{lemma}
\begin{proof} 
Nous travaillons sur $\aclq(A)$. En remplaçant le rang de Morley par
chaque rang local stratifié $D$ \cite[Definition 2.1.6 et
Corollary 2.2.5]{Wa97} dans la preuve de \cite[Lemme 2.3]{Po87}, 
il suffit d'abord de vérifier l'inégalité
$D(\St(h))\ge D(h)$. 
Remarquons que $$D(h') = D(h'/h)=D(hh'/h)=D(hh')=D(hh'/h')=D(h/h') = D(h).$$
Soit $h_2\models\tp(h'/hh')$ avec $h_2\ind_{hh'} h$. Puisque $h'\ind
hh',$
on a $h_2\ind hh'$ et donc $h_2\ind h,h'$~; alors $h\ind h'$ donne
$h\ind h',h_2$ et
$$h\ind h'\inv{h_2}.$$
Comme $h',hh'\equiv h_2,hh',$  on a $h',h\equiv
h_2,hh'\inv{h_2}$. En particulier,
$$hh'\inv{h_2} \equiv h.$$
On a donc $h'\inv{h_2}\in\St(h)$. Ainsi
$$\begin{aligned}D(\St(h))&\ge D(h'\inv{h_2})\ge
  D(h'\inv{h_2}/h_2)=D(h'/h_2)=D(h')=D(h).\end{aligned}$$
On voit facilement que
$$\St(hh')=\St(hh'/h')=\St(h/h')=\St(h)$$
et $$\St(hh')=\St(hh'/h)=h\St(h'/h)h^{-1}=h\St(h')h^{-1}.\qedhere$$
\end{proof}

\begin{remark} Ce lemme se généralise au cas où $H$ est
hyperdéfinissable dans une théorie simple, avec la notion de
stabilisateur correspondante aux théories simples sur la clôture
bornée $\bdd(A)$. On obtient que $\St(hh')=\St(h)$ est commensurable
à un conjugué de $\St(h')$.
\end{remark}

\begin{definition}
Soient $G$ et $H$ deux groupes type-définissables. Un sous-groupe
type-définissable $S$ de $G\times H$ est une {\em endogénie} de $G$
dans $H$ si\begin{itemize}
\item la projection sur $G$ est un sous-groupe $G_S$ d'indice borné, et
\item le {\em co-noyau} $\coker(S)=\{h\in H:(1,h)\in S\}$ est fini.\end{itemize}
Une endogénie induit donc un morphisme de $G_S$ dans
$N_H(\coker(S))/\coker(S)$. Une {\em isogénie} de $G$ vers $H$ est une
endogénie de noyau $\ker(S)=\{g\in G:(g,1)\in S\}$ fini et dont
l'image (la projection sur $H$) est d'indice borné dans $H$.
\end{definition}

\begin{lemma}\label{F:iso} Soient $G$ et $H$ deux groupes
type-définissables dans une théorie stable et $C$ un translaté
type-définissable sur $A$ d'un sous-groupe 
$S\le G\times H$. Supposons qu'il existe un générique $(g,h)$ sur $A$ du
translaté $C$ tel que $h\in\acl(A,g)$. Alors $S$ définit une
endogénie de $\pi_1(S)$ dans $H$, où $\pi_1$ dénote la projection
sur $G$. Si de plus $g$ et $h$ sont interalgébriques sur $A$, alors
l'endogénie est une isogénie de $\pi_1(S)$ vers $\pi_2(S)$, où
$\pi_2$ dénote la projection sur $H$.\end{lemma}
\begin{proof} Supposons que $C$ est un translaté à droite de $S$.
Soit $(g',h')\models\stp(g,h/A)$ avec $g',h'\ind_A
g,h$. Alors $(g'g^{-1},h'h^{-1})$ est générique dans $S$ et
indé\-pen\-dant de $g',h'$ sur $A$. Comme $h'h^{-1}\in\acl(A,g,g')$, on
a $h'h^{-1}\in\acl(A,g'g^{-1},g')$ et donc
$h'h^{-1}\in\acl(A,g'g^{-1})$ car $h'h^{-1}\ind_{A,g'g^{-1}}
g'$. En particulier, le co-noyau
$\coker(S)$ est fini et $S$ est une endogénie de $\pi_1(S)$ vers
$\pi_2(S)$. Si $g\in\acl(A,h)$, on obtient par symétrie que
$\ker(S)$ est fini.\end{proof}

\begin{remark}\label{R:iso} Si $g$ est générique dans un
sous-groupe $G_1\le G$ type-définissable sur $A$, alors $\pi_1(S)$
est d'indice borné dans $G_1$ et $S$ définit une endogénie de $G_1$
dans $\pi_2(S)$. Enfin si $h$ est générique dans un sous-groupe
$H_1\le H$, alors $\pi_2(S)$ a indice borné dans $H_1$.\end{remark}

\begin{remark} Si $G$, $H$ et $S$ sont hyperdéfinissables dans une
théorie simple, on obtient un résultat analogue en remplaçant, dans
les définitions d'endogénie et d'isogénie, fini par
borné.\end{remark}

Jusqu'à la fin de cette partie, on considère une théorie stable $T$
avec un réduit $T_0$, ce réduit ayant élimination géomé\-trique des
imaginaires. On utilisera l'indice $0$ pour indiquer des notions
modèle-théoriques prises au
sens du réduit $T_0$, comme la clôture algébrique ou l'indépendance.
Puisque l'on travaille avec un réduit, la clôture
algébrique dans $T$ sera prise au sens réel.  
Rappelons que l'indépen\-dance au-dessus d'un ensemble $A$
algébriquement clos implique
la $0$-indépen\-dance sur $A$ \cite[Lemme 2.1]{BPW09}. La proposition
suivante donne un
critère pour montrer qu'un sous-groupe connexe type-définissable
d'un groupe $T_0$-type-définissable est lui-même $T_0$-type-définissable. 

\begin{prop}\label{P:expansion}
\`A l'intérieur d'un groupe $T_0$-type-définissable, on considère un
translaté $C$ d'un sous-groupe connexe type-définissable $S$ sur
un ensemble algé\-brique\-ment clos de paramètres réels $A$. Soit $p$
le type générique de $C$ sur $A$ et notons $p_0$ son $T_0$-réduit.
Si $p$ est de rang maximal (par rapport au rang de
Morley, de Lascar ou aux
rangs locaux stratifiés) parmi les complétions de $p_0$
(vu comme $T$-type partiel), alors $S$ est d'indice borné dans le $0$-stabilisateur
$\St_0(p_0)$, qui est son enveloppe $T_0$-type-définissable. 
Si $p$ est l'unique complétion de $p_0$ de rang maximal et est stationnaire,
alors $S$ est $T_0$-type-définissable.\end{prop}
\begin{proof} 
Comme les hypothèses de la proposition se préservent par extension des
paramètres, en prenant un représentant du translaté de $S$, on peut se
réduire au cas où $p$ est le générique de $S$. 

Remarquons que $S=\St(p)\subseteq\St_0(p_0)$, 
puisque l'indépendance implique la $0$-indépendance. Comme $S$ est 
un sous-groupe connexe $A$-invariant qui contient $p$, l'enveloppe
$T_0$-type-définissable $\bar S$ est aussi un sous-groupe connexe
$A$-invariant et contient $p_0$. Donc, il contient $\St_0(p_0)\subseteq
p_0p_0^{-1}$. Il suit que $\bar S=\St_0(p_0)$, qui est alors
$0$-connexe. En particulier, le type $p_0$ est l'unique $0$-type
générique de $\bar S$.

Rappelons qu'une formule est générique dans un groupe stable si un
nombre fini de ses translatés recouvre ce groupe. Ainsi,
une $0$-formule est générique pour $\bar S$ (au sens de $T_0$ comme
au sens de $T$) si et seulement si elle l'est dans $p_0$. Les types
génériques de $\bar S$ sont donc précisément les complétions de
$p_0$ de rang maximal.

Ainsi, si $p$ est de rang maximal parmi les complétions de $p_0$,
alors $p$ est générique pour $\bar S$, ce qui entraîne que l'indice de
$S$ dans $\bar S$ est borné. Si de plus $p$ est l'unique complétion de
rang maximal de $p_0$, alors $p$ est l'unique type générique de
$\bar S$, qui est donc connexe (au sens de $T$) et égal à $S$.
\end{proof}

Pour terminer ces préliminaires, on rappelle un fait
qui sera utilisé à plusieurs reprises par la suite. 
Soient $G$ un groupe connexe type-définissable sur $\emptyset$ dans 
$T$, et $a$ et $b$ deux génériques indépendants de $G$.

\begin{fact}\label{F:reduitgroupe}\mbox{}\cite[Lemme 3.4 et Remarque
3.5]{BPW09}
Après adjonction au langage des éléments d'une suite de Morley $D$ du
générique de $G$ au-dessus de $a,b$, on pose
$$\begin{aligned} a_1&=\acl_0(\acl(b),\acl(ab))\cap\acl(a),\\
b_1&=\acl_0(\acl(a),\acl(ab))\cap\acl(b),\\
(ab)_1&=\acl_0(\acl(a),\acl(b))\cap\acl(ab).\end{aligned}$$
Alors $a_1$, $b_1$ et $(ab)_1$ sont indépendants deux-à-deux, chacun est
$0$-algébrique sur les deux autres, et de plus
$$\acl(b),\acl(ab)\indi0_{a_1} \acl(a).$$
\end{fact}

Nous allons en déduire une variante de \cite[Theorem 3.1]{BPW09} qui est plus
explicite sur l'image du morphisme. Rappelons que le préfixe $*$ indique que le
nombre de variables est infini.
\begin{theorem}\label{T:morphisme} Soit $G$ un groupe
connexe type-définissable sur $\emptyset$ dans  une théorie stable $T$ qui a un
réduit $T_0$  éliminant géomé\-trique\-ment les imaginaires. Alors il
existe un ensemble $A$ de paramètres algébriquement clos et un
homomorphisme définissable $\phi^*$ de $G$ vers un groupe $T_0$-$*$-interprétable
$H^*$ sur $A$, tel que pour chaque couple de génériques indépendants $a$ et $b$
de $G$ sur $A$ on ait 
$$\acl(b,A),\acl(ab,A)\indi0_{\phi^*(a),A}\acl(a,A).$$
De plus, s'il existe un homomorphisme définissable $\phi$ de $G$ dans un
groupe $T_0$-inter\-pré\-table sur un ensemble $B$ algébriquement clos tel que
pour un générique $a$ de $G$ sur $B$, le $0$-rang de Lascar $U_0(\phi(a)/B)$ est
maximal et fini, alors pour tout $C\supseteq B$ algébriquement clos, on a 
$$\acl(b,C),\acl(ab,C)\indi0_{\phi(a),C}\acl(a,C),$$
pour chaque couple de génériques indépendants $a$ et $b$ de $G$ au dessus de $C$.
\end{theorem}
\begin{proof} Prenons trois génériques indépendants $a,b$ et $c$ de $G$ sur une
suite $D$ de Morley du générique de $G$ et posons
$$a_1=\acl_0(\acl(b,D),\acl(ab,D))\cap\acl(a,D).$$
Comme $a,b,ab\equiv_{D}^{stp} a^{-1},c^{-1},(ca)^{-1}$, on a
$$a_1=\acl_0(\acl(c,D),\acl(ca,D))\cap\acl(a,D).$$
Pour $b_1,(ab)_1,c_1,(ca)_1$ et $(cab)_1$ définis de manière analogue, le fait
\ref{F:reduitgroupe} implique que le $6$-uple
$(a_1,b_1,c_1,(ab)_1,(ca)_1,(cab)_1)$ est un quadrangle $0$-algébrique~; le
théorème de configuration de groupe \cite{eHPhD}
(voir aussi \cite{Bou89} ou \cite[Theorem 5.4.5 et Remark 5.4.9]{Pi96}), dont la
preuve reste valable pour les uples infinis,
donne l'existence d'un groupe $T_0$-$*$-interprétable $0$-connexe $H^*_0$ sur un
ensemble algébriquement clos $A_0\supseteq D$ de para\-mètres indépendant de
$a,b,c$ sur $D$. Par transitivité, on a $a,b,c,D\ind A_0$. De plus, les uples
$a_1$ et $b_1$ sont respectivement $0$-interalgébriques avec des $0$-génériques
$h$ et $h'$ de $H^*_0$ sur $A_0$, et $(ab)_1$ est $0$-interalgébrique avec $hh'$
sur $A_0$. Comme $a$, $b$ et $ab$ sont deux-à-deux indépendants sur $A_0$,
les couples 
$$(a,h),\qquad(b,h')\qquad\text{et}\qquad(ab,hh')$$
les sont aussi. D'après le lemme \ref{F:iso} et la
remarque \ref{R:iso}, le stabilisateur $\St(a,h/A_0)$ définit une endogénie
$\phi^*_0$ de $G$ dans $H^*_0$. On peut supposer que $\phi^*_0$ est un homomorphisme,
en remplaçant $H^*_0$ par $N_{H^*_0}(\coker(\phi^*_0))/\coker(\phi^*_0)$, qui est
également $T_0$-$*$-interprétable. Notons que les trois couples ont même type
sur $A_0$. Ainsi $(a,h)$ est générique dans $\St(a,h/A_0)$, et $\phi^*_0(a)=h$.

Comme $A_0\ind a,b,c,D$ implique $A_0\indi0\acl(a,b,c,D)$ d'après \cite[lemme
2.1]{BPW09}, et donc $A_0\indi0_{a_1}\acl(a,b,c)$, l'indépendance
$$\acl(b),\acl(ab)\indi0_{a_1}\acl(a)\quad\text{donne}\quad\acl(b),
\acl(ab)\indi0_{a_1,A_0}\acl(a).$$
Par $0$-interalgébricité, on obtient
$$\acl(b),\acl(ab)\indi0_{\phi^*_0(a),A_0}\acl(a).$$
Par récurrence, on construit une suite $(A_n,H^*_n,\phi^*_n:n<\omega)$ avec
les propriétés suivantes (en posant $A_{-1}=\emptyset$)~:
\begin{itemize}
\item l'ensemble $A_n\supseteq A_{n-1}$ est algébriquement clos et indépendant
de $a,b,c$,
\item le groupe $H^*_n$ et l'homomorphisme 
$\phi^*_n:G\to H^*_n$ sont $T_0$-$*$-interprétables sur $A_n$,
\item on a l'indépendance
$$\acl(b,A_{n-1}),\acl(ab,A_{n-1})\indi0_{\phi^*_n(a),A_n}\acl(a,A_{n-1}).$$
\end{itemize}
\medskip
Posons $A=\bigcup_{n<\omega}A_n$, $H^*=\prod_{n<\omega}H^*_n$ et 
$\phi^*=\prod_{n<\omega}\phi^*_n:G\to H^*$. Puisque chaque formule dans $\tp(\acl(b,A),\acl(ab,A)/\acl(a,A))$ n'utilise qu'un
nombre fini de paramètres et ne dévie donc pas sur $\phi^*_n(a),A_n$ pour $n$
suffisamment grand, on conclut que 
\begin{equation*}\tag{\dag}\acl(b,A),\acl(ab,A)\indi0_{\phi^*(a),A}\acl(a,A).
\end{equation*}

Supposons maintenant que $\phi$ soit un homomorphisme interprétable de
$G$ vers un groupe $T_0$-interprétable $H$ sur un ensemble $B$ algébriquement
clos tel que pour $a$ générique dans $G$ sur $B$ le $0$-rang de
Lascar $U_0(\phi(a)/B)$ soit maximal et fini.
Pour $C\supseteq B$ algébriquement clos, construisons,
comme précédemment et au-dessus de $C$, un
homomorphisme définissable $\phi^*$ de $G$ vers un groupe
$T_0$-$*$-interprétable $H^*$ sur $A$ ($\supseteq C$).
Par stabilité, le groupe $H^*$ est une limite projective
$\varprojlim\pi_i(H^*)$, où chaque $\pi_i(H^*)$ est un groupe 
$T_0$-interprétable. Si $a$, $b$ sont deux génériques indépendants
de $G$ sur $A$, alors $\phi(a)\in\acl(a,A)$, $\phi(b)\in\acl(b,A)$ et
$\phi(ab)\in\acl(ab,A)$. Ainsi l'indépendance $(\dag)$ donne
$$\phi(a)=\phi(ab)\cdot\phi(b)^{-1}\in\acl(a,A)\cap\acl_0(\acl(b,A),\acl(ab,
A))=\acl_0(\phi^*(a),A).$$
Il existe donc $i$ tel que $\phi(a)\in\acl_0((\pi_j\circ\phi^*)(a),A)$ pour tout 
$j\ge i$. Inversement, comme $(\phi,\pi_j\circ\phi^*)$
est aussi un homomorphisme interprétable de $G$ vers un groupe
$T_0$-interprétable, on a
$$U_0(\phi(a)/A)\le U_0(\phi(a),(\pi_j\circ\phi^*)(a)/A)\le
U_0(\phi(a)/A)$$
par maximalité du rang. Ainsi $U_0((\pi_j\circ\phi^*)(a)/\phi(a),A)=0$ pour tout
$j\ge i$, et $\phi^*(a)\in\acl_0(\phi(a),A)$, ce qui donne
$$\acl(b,A),\acl(ab,A)\indi0_{\phi(a),A}\acl(a,A).$$
Puisque $A\ind a$ on a $A\ind_{C}\acl(a,C)$, et donc
$A\indi0_{C}\acl(a,C)$ par \cite[lemme 2.1]{BPW09}. 
Alors $\phi(a)\in\acl(a,B)\subseteq\acl(a,C)$ implique 
$$\phi(a),A\indi0_{\phi(a),C}\acl(a,C),$$
et par transitivité
$$\acl(a,C)\indi0_{\phi(a),C}\acl(b,C),\acl(ab,C).\qedhere$$
\end{proof}

\section{Die liebe Farbe}\label{S:liebe}
Dans cette partie nous rappelons quelques propriétés du mauvais corps
obtenu dans \cite{BHPW06} qui seront utilisées pour la démonstration
du théorème A. Il s'agit d'un corps $K$ algébriquement clos de
caractéristique nulle muni d'un sous-groupe propre multiplicatif
divisible et sans torsion, noté $\U$ et coloré en vert, tel que
$\RM(K)=2$ et $\RM(\U)=1$ (où $\RM(X)$ désigne le rang de Morley de
$X$, qui dans tout groupe de rang de Morley fini est égal au rang de
Lascar). On peut alors voir $\U$ comme un $\Q$-espace vectoriel qui
est fortement minimal avec toute la structure induite. Nous
travaillons dans un langage relationnel, à l'exception de la loi du
groupe multiplicatif et, pour chaque $q\in\Q$, de la multiplication
linéaire par $q$ sur les éléments verts. \`A l'élément zéro près, nos
structures sont donc des sous-groupes multiplicatifs d'un corps de
caractéristique nulle dont le sous-groupe vert est divisible sans
torsion. Pour deux structures $A$ et $B$, on note $A\sse B$ la
structure engendrée, qui est, modulo l'élément $0$, le plus petit
sous-groupe multiplicatif les contenant et clos par racines vertes.

Pour deux structures $A\subseteq B$ avec $B$ finiment engendrée sur
$A$, on définit une prédimension relative
$$\dd(B/A) = 2 \,\tr(B/A) - \ld_\Q(\U(B)/\U(A))~;$$
lorsque $A$ est vide, on l'omet. On dit que $A$ est {\em
 autosuffisante} dans $B$, noté $A\le B$, si $\dd(B_0/A)\ge0$ pour
tout $A\subseteq B_0\subseteq B$. Une extension autosuffisante $A\le
B$ est {\em minimale} s'il n'y a pas de structure intermédiaire propre
autosuffisante dans $B$.

Le corps $K$ est caractérisé par les conditions suivantes:\begin{itemize} 
\item $\dd(A)\ge 0$ pour tout $A\subset K$ finiment engendré.
\item Si $A\le K$ et $A\le B$ est une extension minimale , alors il existe un 
plongement de $B$ dans $K$ sur $A$ avec image autosuffisante.
\item Si $A\le K$ et $A\le B$ avec $\dd(B/A)=0$, alors il n'y a qu'un nombre
fini de plongements de $B$ dans $K$ sur $A$.\end{itemize}
Rappelons que $\delta$ est {\em sous-modulaire}~: pour $A$ et $B$ dans
$K$ finiment engendrées,
$$\delta(A\sse B)+\delta(A\cap B) \leq \delta(A)+\delta(B).$$ 
En particulier, l'intersection de deux sous-structures autosuffisantes
l'est aussi. Donc pour tout  $A\subseteq K$, il existe une plus petite
sous-structure autosuffisante de $K$ contenant $A$, sa {\em clôture
  autosuffisante} $\sscl{A}$, qui est algébrique sur $A$ dans la
théorie $T$ du mauvais corps $K$ et  $A$-invariante en tant
qu'ensemble. Si $A$ est finiment engendré, alors $\sscl{A}$ l'est
aussi. De plus, 
$$\RM(b/A)=\dd(\sscl{Ab}/\sscl{A}).$$ 
L'indépendance $a\ind_Cb$ au sens de $T$ pour deux uples $a$ et $b$
au-dessus de $C=\sscl{aC}\cap\sscl{bC}$ est caractérisée par les
propriétés équivalentes suivantes~:\\
\begin{itemize}
\item $\delta(\sscl{abC}/\sscl{bC})=\delta(\sscl{aC}/C)$\\
\item $\sscl{aC}\indi 0_C\sscl{bC}$, $\sscl{abC}=\sscl{aC}\sse\sscl{bC}$ et
$\U(\sscl{abC})=\U(\sscl{aC})\cdot\U(\sscl{bC})$,\\
\end{itemize}
où l'indice $0$ fait référence au réduit au pur langage des
anneaux, donc 
à la théorie $T_0$ des corps algébriquement clos de caractéristique
$0$. Notons que $T_0$ élimine les imaginaires, ce qui nous permet
d'utiliser les résultats de la section précédente. De plus, si $A$ et
$B$ sont autosuffisants et indépendants sur leur intersection, alors
la structure engendrée $A\sse B$ est, modulo $0$, le produit des groupes 
$A$ et $B$. La théorie $T$ est relativement
CM-triviale au-dessus de $T_0$ par rapport à la clôture
autosuffisante \cite[Théorème 6.4]{BPW09}.

\begin{remark}\label{R:ind}
Soient $A\subseteq B$ autosuffisants et $C\subseteq\acl(A)$
autosuffisant tel que $B\cap C=A$. Alors, on a $B\indi 0_A C$.
\end{remark}
\begin{proof}
Par le caractère local de la déviation, on peut supposer que $C$ est finiment
engendré au-dessus de $A$. Comme $C$ est $T$-algébrique sur $A$, on a
$B\ind_AC$~; la caractérisation de l'indépendance implique alors
$B\indi0_AC$.\end{proof}

Pour deux structures
$A\subseteq B$, on appelle une {\em base verte} de $B$ sur $A$ tout
uple de $\U(B)$ qui complète une base linéaire de $\U(A)$ en une base linéaire
de $\U(B)$~; notons qu'elle sera linéairement indépendante sur $A$
car toute combinaison linéaire de points verts est verte.
En particulier, une base verte de $B$ est une base linéaire de $\U(B)$. 

Terminons par quelques remarques utiles sur le corps $K$ :
\begin{remark}\label{R:verts}
Supposons que la sous-structure $A$ est autosuffisante.
\begin{enumerate}
\item\label{R:verts:pasdenouveauxverts}
La sous-structure $\acl_0(A)$ est à nouveau autosuffisante et elle ne 
contient pas de nouveaux points colorés.
\item\label{R:verts:algebriquesurverts} L'élément $b$ est algébrique  
sur $A$ si et seulement s'il existe $B\supseteq Ab$ finiment engendré sur 
$A$ tel que $\dd(B/A) = 0$. Si $\bar x$ est une base verte de $B$ sur $A$ 
alors $\tr(B/A) = \tr(\bar x/A) = |\bar x|/2$.
En particulier, l'ensemble $B$ est inclus dans $\acl_0(A,\U(B))$~; si
$A$ est vert, on a même $B \subseteq \acl_0(\U(B))$. 
\item\label{R:verts:sommedeuxverts} Tout point blanc est
somme de deux points verts et chacun de ces couples est algébrique sur ce point
blanc. Par compacité, l'ensemble de ces couples est fini. 
Supposons que $A$ est engendrée par un ensemble fini $\Theta$ et considérons  la
sous-structure $B$ (finiment)  engendrée par une base verte de $A$ et l'ensemble
des couples verts dont la somme est un élément blanc de $\Theta$. Alors $B$  est
autosuffisante et contenue dans $\acl(A)$~:~la sous-structure $A\sse B$ l'est
car $\dd (A\sse B)=\dd(A)$. Puisque  $\U(A\sse B) = \U(B)$, on en déduit
l'autosuffisance de $B$. 
\end{enumerate}
\end{remark}

\section{Cherchez le tore}\label{cherchez}
Dans cette partie, nous allons démontrer le théorème A énoncé dans
l'introduction. Pour ce faire, nous allons décomposer le groupe en
une partie $0$-algébrique et une partie $0$-transcendante (où l'indice
$0$ fait toujours référence au réduit de pur corps algébriquement
clos). En utilisant cette décomposition, on construira un groupe algébrique.
Finalement, on exhibera une isogénie non triviale de notre groupe 
vers une section du groupe algébrique construit.

Rappelons qu'un corps de rang de Morley fini élimine les imaginaires
\cite{Wa01} et donc tout groupe interprétable est définissablement
isomorphe à un groupe définissable. En particulier, le quotient
$K^*\!/\U$ est définissablement isomorphe à un
groupe définis\-sable. Hélas, pour deux uples réels $a$ et $b$ en
bijection avec deux génériques indépen\-dants de $K^*\!/\U$, l'élément
$a_1$  décrit dans le fait \ref{F:reduitgroupe} est vide. Cet exemple
sera notre fil conducteur pour faire une analyse plus précise des
groupes définissables. 

Fixons donc un groupe infini connexe $G$ définissable 
dans un corps vert collapsé $(K,\U)$. Nous ajoutons aux paramètres de
la théorie ceux nécessaires pour définir $G$ ainsi que la clôture
algébrique d'une suite de Morley $D$ du générique de $G$. Nous
supposons en particulier que $G$ est définissable sur $\emptyset$. 

Prenons $a$, $b$ et $c$ trois points génériques de $G$ indépendants,
et posons $$a_1=\acl(a)\cap \acl_0(\acl(b),\acl(ab)).$$

Alors $a_1$ est autosuffisant comme intersection de deux sous-structures
autosuffisantes. Notons que $\acl(a)$ et $a_1$ peuvent avoir degré de
transcendance infini. Par ailleurs, on a également
$a_1=\acl(a)\cap \acl_0(\acl(c),\acl(ca))$ par \cite[Lemme
3.4]{BPW09}. 

Le fait \ref{F:reduitgroupe} donne
$$\acl(a)\indi0_{a_1}\acl(b),\acl(ab)\quad\text{ainsi que}\quad 
\acl(a)\indi0_{a_1} \acl_0(\acl(c),\acl(ca)).$$

\begin{lemma}\label{L:a_to_abar}
Il existe un uple autosuffisant $0$-algébriquement clos $\a$  tel
que~:
\begin{itemize}
\item $\a$ contient $a$, $a^{-1}$ et $a_1$,
\item $\a$ est algébrique sur $a$,
\item $\a$ est $0$-algébrique sur sa base verte,
\item le degré de transcendance  $\tr(\a/a_1)$ est fini.
\end{itemize}
On obtient de même des uples $\b$,
$\ab$, $\cc$, $\ca$ et $\cab$, et le diagramme suivant~: 

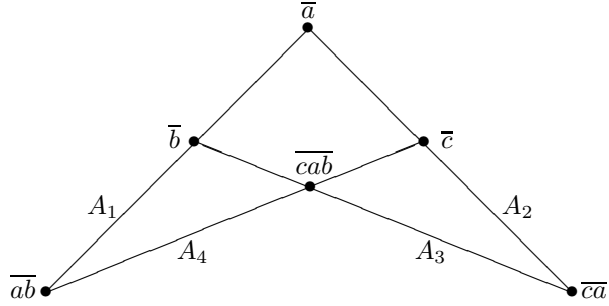
\begin{center}
\begin{picture}(200,110)
\put(0,0){\line(1,1){100}}
\put(200,0){\line(-1,1){100}}
\put(0,0){\line(5,2){143}}
\put(200,0){\line(-5,2){143}}
\put(-13,-2){$\ab$}
\put(203,-2){$\ca$}
\put(97,104){$\a$}
\put(47,55){$\b$}
\put(150,55){$\cc$}
\put(95,45){$\cab$}
\put(-2,-2){$\bullet$}
\put(197,-2){$\bullet$}
\put(97,98){$\bullet$}
\put(98,38){$\bullet$}
\put(54,55){$\bullet$}
\put(141,55){$\bullet$}
\put(16,30){$A_1$}
\put(173,30){$A_2$}
\put(50,13){$A_4$}
\put(140,13){$A_3$}
\end{picture}
\end{center}
\vskip4mm
\noindent où $A_1$ est l'ensemble autosuffisant $\acl_0(\sscl{\a,\b,\ab})$ 
(et de même pour les autres droites
$A_i$). De plus~:\begin{itemize}
\item $\a=A_1\cap A_2$ et $A_1\ind_a A_2$,
\item $A_1\indi0_{\a}A_2$, 
\item la structure  $A_1\sse A_2$ est autosuffisante et $\U(A_1\sse 
A_2)=\U(A_1)\cdot\U(A_2)$, où  $\U(A_1)$ et $\U(A_2)$ sont en somme
directe au-dessus de $\U(\a)$ en tant que $\Q$-espaces vectoriels.
\end{itemize}
\end{lemma}

\begin{proof}
Considérons un uple réel $a_E$ représentant
l'ensemble $\{a,a^{-1}\}$, où l'élément $a^{-1}$ est l'inverse de $a$ au
sens du groupe $G$. 
Alors $a$, $a^{-1}$ et $a_E$ sont tous interalgébriques au sens de la 
théorie $T$. 

La clôture autosuffisante $\sscl{a,a^{-1}}$ est une
sous-structure finiment engendrée et algébrique sur $a_E$.
Soit $\Theta$ un ensemble fini $a_E$-définissable contenu dans 
$\sscl{a,a^{-1}}$ et contenant $a$, $a^{-1}$ ainsi qu'une base verte de
$\sscl{a,a^{-1}}$. Alors la structure engendrée par $\Theta$ est égale à
$\sscl{a,a^{-1}}$. 
Soit $a'_2$ l'ensemble fini $a_E$-définissable constitué de $\U(\Theta)$ 
ainsi que tous les couples de points verts dont la somme est un point blanc de
$\Theta$. Par la
remarque \ref{R:verts}(\ref{R:verts:sommedeuxverts}), l'uple vert fini
$a'_2$ engendre une sous-structure verte autosuffisante incluse dans
$\acl(a)$. 
De plus, les éléments $a$ et $a^{-1}$ sont $0$-algébriques sur $a'_2$.

En utilisant la définissabilité de $a'_2$ sur $a_E$ et le fait que tous 
les points ont le même type, on décompose de la même
façon les points $b$, $ab$, $c$, $ca$ et $cab$. 

Posons $$\begin{array}{ll}
A_1 = \acl_0(\sscl{a_1,a'_2,b_1,b'_2,ab_1,ab'_2}), &
A_2 = \acl_0(\sscl{a_1,a'_2,c_1,c'_2,ca_1,ca'_2}), \\
A_3 = \acl_0(\sscl{b_1,b'_2,ca_1,ca'_2,cab_1,cab'_2}), &
A_4 = \acl_0(\sscl{ab_1,ab'_2,c_1,c'_2,cab_1,cab'_2}),
\end{array}$$
où la multiplication est prioritaire sur les indices~: par exemple
$cab_1$ signifie $(cab)_1$ et $ab'_2$ signifie $(ab)'_2$. 

Par la
remarque \ref{R:verts}(\ref{R:verts:pasdenouveauxverts}), les
sous-structures $A_1$, $A_2$, $A_3$ et $A_4$ sont toutes 
autosuffisantes.

Si l'on pose $$
\a=\acl(a)\cap A_1,\quad\b=\acl(b)\cap A_1 \quad \text{et}\quad
\ab=\acl(ab)\cap A_1,$$
qui sont autosuffisants et $0$-algébriquement clos, alors
$a_1,a'_2\subseteq\a=\acl_0(\a)$,  donc
$a\subseteq\a\subseteq\acl(a)$.
Il suit que $A_1=\acl_0(\sscl{\a,\b,\ab})$. De plus, on a $\a=\acl(a)\cap 
A_2$, car $(b,ab)$ et 
$(c^{-1},(ca)^{-1})$ ont le même type sur $a_E$,  donc 
$A_1\equiv_{a_E}A_2$. On définit $\cc$, $\ca$ et $\cab$ de façon analogue. 

La remarque \ref{R:ind} donne $A_1\indi0_{\a}\acl(a)$. Puisque $a'_2
\subseteq \a\subseteq\acl(a)=\acl(a'_2)$, on a $\dd(\a/a'_2)=0$.
Donc $\a = \acl_0(\U(\a))$ d'après la remarque \ref{R:verts}
(\ref{R:verts:algebriquesurverts}).

Le  fait \ref{F:reduitgroupe} donne que
$\acl_0(a_1,b_1,ab_1)=\acl_0(a_1,b_1)$, qui  est 
autosuffisant. 
Comme $a'_2,b'_2,ab'_2$ sont finis, le dégré de transcendance
$\tr(A_1/a_1,b_1)$ est fini. Puisque
$$\acl(a)\indi0_{a_1} \acl(b),\acl(ab),$$ 
on conclut que le degré de transcendance
$$\tr(\a/a_1)=\tr(\a/a_1,b_1,b'_2,ab_1,ab'_2)\le\tr(A_1/a_1,b_1),$$
est aussi fini.

Puisque $b$ et $c$ sont indépendants sur $a$, on a
$A_1\ind_a A_2$. Donc
$$\a\subseteq A_1\cap A_2\subseteq A_1\cap\acl(a)=\a.$$
Ainsi $A_1\cap A_2=\a$ et 
$$A_1\ind_{\a} A_2,$$
car $a$ et $\a$ sont interalgébriques  (au sens $T$). Il suit de la
caractérisation de l'indé\-pen\-dance que  $A_1\sse A_2$ est 
autosuffisante, 
$$A_1\indi0_{\a}A_2\quad\text{et}\quad\U(A_1\sse 
A_2)=\U(A_1)\cdot\U(A_2),$$
où $\U(A_1)$ et $\U(A_2)$ sont en somme directe au-dessus de $\U(\a)$
en tant que $\Q$-espaces vectoriels.
\end{proof}

Nous venons donc de remplacer les points génériques du groupe $G$ par
des uples autosuffisants interalgébriques, qui sont $0$-algébriques sur 
leurs bases vertes. Cela nous permettra de construire une 
$0$-configuration de groupe à partir de points colorés. Quoique le 
diagramme précédent ne donne pas une $0$-configuration de groupe, les 
droites sont indépendantes sur leurs intersections communes. 
\begin{lemma}\label{L:preind}
$$A_1 \indi0_{\a,\b} A_2,A_3.$$
\end{lemma}
\begin{proof}Vérifions d'abord que
$$A_1 \indi 0_{\a\sse\b} \acl(\a)\sse \acl(\b).$$
Comme $\a$ et $\b$ sont indépendants, les structures $\a\sse \b$ et
$\acl(\a)\sse \acl(\b)$ sont autosuffisantes. Puisque
$\acl(\a)\sse \acl(\b)\subseteq\acl(\a,\b)$,  il suffit de montrer que  
$A_1 \cap(\acl(\a)\sse\acl(\b)) = \acl_0(\a, \b)$, par la remarque
\ref{R:ind}. L'ensemble $A_1 \cap(\acl(\a)\sse\acl(\b))$ est
$0$-algébrique sur ses points verts et 
$\a\sse \b$ par la remarque
\ref{R:verts}(\ref{R:verts:algebriquesurverts}). Or, son degré
de transcendance sur $\a\sse \b$ est fini,  
donc une (toute) base verte $e$ de cet ensemble sur $\a\sse\b$ est
finie. Par hypothèse, elle s'écrit comme $e = e_a\cdot e_b$, avec $e_a
\in \acl(\a)$ et $e_b \in \acl(\b)$, tous les deux verts, où le produit 
est pris coordonnée
par coordonnée. Notons que
$$\begin{array}{r@{}lcl@{}}
e&\indi 0_{\a,\b} e_a&\text{car}&A_1\indi 0_{\a}\acl(\a),\\
e&\indi 0_{\a, \b} e_b&\text{car}&A_1\indi 0_{\b}\acl(\b),\mbox{ et }\\
e_a&\indi 0_{\a , \b} e_b&\text{car}&\acl(\a)\indi
0\acl(\b)\mbox{ par
indépendance au sens de $T$}.\end{array}$$
Donc, d'après le lemme \ref{L:stab}, l'élément $e$ est $0$-générique
dans un translaté d'un sous-groupe algébrique connexe multiplicatif. Un
tel sous-groupe de $(K^*)^n$ est un tore, défini par des relations
$\Q$-linéaires. La dimension linéaire d'un point générique coïncide donc avec le degré de transcendance \cite{yM03}.

Notons que 
$$\dd(e/\acl_0(\a, \b))= \dd(A_1
\cap(\acl(\a)\sse\acl(\b))/\acl_0(\a, \b))=0.$$
Ainsi 
$$0=\dd(e/\acl_0(\a,\b))=2\tr(e/\a,\b)-\ld_\Q(e/\acl_0(\a,\b))=\tr(e/\a,\b),$$
et donc $e\in\acl_0(\a, \b)$.
Alors, $$A_1\cap(\acl(\a)\sse\acl(\b)) \subseteq \acl_0(e,\a \sse \b) 
\subseteq\acl_0(\a,\b).$$

Pour terminer, observons que l'on travaille depuis le début de cette partie
au-dessus d'une suite de Morley $D$ du générique de $G$
utilisée pour le fait \ref{F:reduitgroupe}. Une relation
$0$-algébrique entre $A_1$ et $A_2,A_3$ n'utilise qu'une partie infinie 
propre $D'$ de $D$,   c'est-à-dire entre les ensembles $A'_1$ et
$A'_2,A'_3$ correspondants à la même construction au-dessus de la suite de
Morley $D'$.  On peut envoyer $c$ sur un élément de $D\setminus D'$ sur 
$D',a,b$ pour obtenir une copie de $A'_2\sse A'_3$ dans
$\acl(a)\sse \acl(b)$ au-dessus de $A'_1$.
Donc cette relation $0$-algébrique se produit déjà entre 
$A_1$ et $\acl(a)\sse \acl(b)$ et, par transitivité, entre $A_1$ et 
$\a\sse \b$.
\end{proof}

\begin{corollary}\label{C:ind} La structure $A_1\sse A_2\sse A_3$ est
autosuffisante. On a
$$\U(A_1\sse A_2\sse A_3)=\U(A_1)\cdot\U(A_2)\cdot\U(A_3).$$
En outre, ces sous-espaces sont en somme directe au-dessus de
$\U(\a)\cdot\U(\b)\cdot\U(\ca)$.
\end{corollary}
\begin{proof}
Puisque
$$\a\sse\b\subseteq A_1\cap
(A_2\sse A_3)\subseteq A_1\subseteq\acl(\a,\b)$$
sont tous des ensembles autosuffisants, on obtient
$$0\le\dd(A_1/A_1\cap (A_2\sse
A_3))\le\dd(\acl(\a,\b)/\a\sse\b)=0.$$
On a donc égalité~; par sous-modularité et autosuffisance de $A_2\sse
A_3$ on obtient
$$0\le\dd(A_1\sse A_2\sse A_3/A_2\sse A_3)\leq\dd(A_1/A_1\cap (A_2\sse
A_3))=0.$$
Comme la structure  $A_2\sse A_3$ est autosuffisante, on en déduit que 
$A_1\sse A_2\sse A_3$ l'est également.

Puisque $\a\sse\b$ est autosuffisant,
la remarque \ref{R:verts}(\ref{R:verts:pasdenouveauxverts})
donne $\U(\acl_0(\a,\b))=\U(\a\sse\b)$. Enfin, comme
$$\dd(A_1\sse A_2\sse
A_3/\a\sse\b)=\RM(G)=\dd(A_1/\a\sse\b)+\dd(A_2\sse A_3/\a\sse\b)$$
et
$$A_1\indi0_{\a,\b}A_2,A_3\,,$$ on conclut que $$\U(A_1\sse A_2\sse
A_3)=\U(A_1)\cdot\U(A_2 \sse A_3),$$
et que $\U(A_1)$ et $\U(A_2\sse A_3)$ sont en somme directe au-dessus
de $\U(\a\sse\b)=\U(\a)\cdot \U(\b)$. Le fait que $\U(A_2\sse A_3)$ est la 
somme directe
de $\U(A_2)$ et $\U(A_3)$ au-dessus de $\U(\ca)$ nous permet de conclure.
\end{proof}

Le degré de transcendance $\tr(\bar a/a_1)$ est fini d'après le lemme 
\ref{L:a_to_abar}. Considérons alors un uple fini $a_2$ de points verts de $\a$
maximal $0$-transcendant sur $a_1$. Comme $\bar a=\acl_0(\U(\bar a))$, on
a $\a=\acl_0(a_1a_2)=\acl_0(\sscl{a_1a_2})$, car
$\a$ est autosuffisant.  La remarque
\ref{R:verts}$(\ref{R:verts:pasdenouveauxverts})$ donne que $\U(\a) = 
\U(\sscl{a_1a_2})$. 

Nous allons maintenant compléter $a_2$ en une base verte de $\a$ 
sur $a_1$ de la façon la plus transcendante possible.
\begin{lemma}\label{L:t_a_trans}
Il existe un uple vert $t_a$ fini et $0$-transcendant sur $a_1$ qui 
complète $a_2$ en une base
verte de $\a$ sur $a_1$.\end{lemma}
\begin{proof}
Soit $t_a$ complétant $a_2$ en une base verte de $\a$ sur
$a_1$. Comme $a_1$ est autosuffisant, toute sous-structure de $\a$
engendrée sur $a_1$ par $n$ éléments de $a_2,t_a$ a degré de
transcendance au moins $n/2$ sur $a_1$. En particulier, 
$$|a_2| = \tr (\a/a_1) \geq |t_a|.$$
On peut alors transformer linéairement les éléments de $t_a$ avec les
points de $a_2$ pour que $t_a$ soit également $0$-transcendant sur $a_1$.
\end{proof}

\begin{remark}\label{R:base}
Rappelons que $t_a$ est $0$-algébrique sur $a_1a_2$. Nous obtenons
des bases vertes comme ci-dessus pour chaque point du diagramme. Les
indépendances 
$$\acl(b)\indi
0\acl(ab)\quad\text{et} \quad \acl(a)\indi0_{a_1}\acl(b),\acl(ab)$$ 
impliquent que $a_2,b_2,ab_2,t_a,t_b,t_{ab}$ est 
une base linéaire de $\U(\a)\cdot\U(\b)\cdot\U(\ab)$ au-dessus de
$\U(\acl_0(a_1,b_1,ab_1))=\U(\acl_0(a_1,b_1))=\U(a_1)\cdot\U(b_1)$.\end{remark}

\begin{lemma}\label{L:modify_t_a} 
Les uples
$$t'_a=\frac{t_c}{t_{ca}},\quad t'_b=\frac{t_{ca}}{t_{cab}},\quad\mbox{et}\quad
t'_{ab}=\frac{t_c}{t_{cab}}$$
(où la division est coordonnée par coordonnée)
satisfont~:\begin{enumerate}
\item $t'_a\cdot t'_b=t'_{ab}$;
\item $t'_a$ est $0$-transcendant sur $\a,\ca$ et de même sur $\a,\cc$.
\end{enumerate}
\end{lemma}
\begin{proof}
L'égalité est évidente. Le fait \ref{F:reduitgroupe}
entraîne que $$t_c\indi0_{c_1}\a,\ca.$$
Comme $t_c$ est $0$-transcendant sur $c_1$, il l'est aussi sur
$\a,\ca$. Puisque
$t_{ca}\in\acl_0(\ca)$, on voit que $t'_a$ est $0$-transcendant sur
$\a,\ca$. \end{proof}

Soit $t_1$ une base linéaire de $\U(A_1)$ sur
$\U(\a)\cdot\U(\b)\cdot\U(\ab)$, et de même pour les autres droites.
Notons que $t_1$ est fini car $\tr(A_1/a_1,b_1)$ l'est. 
Alors $$a_2,b_2,ab_2,t_a,t_b,t_{ab},t_1$$ est une base verte de $A_1$
sur $\acl_0(a_1,b_1)$.
Nous allons transformer ces bases afin qu'elles satisfassent les
conditions du lemme \ref{L:stab}.
\begin{prop}\label{modify} Après une éventuelle transformation
linéaire des uples $t_i$, on a~:\begin{enumerate}
\item $t_1\cdot  t_4 = t_2 \cdot t_3$,
\item $t_i$ et $A_i$ sont $0$-algébriques sur les trois points de la
 droite $A_i$,
\item $|t_2|+|t_c| = |c_2|$ et
$\RM(\cc/c_1)=\dd(\cc/c_1)=|t_2|$.
\item $t_2$ est $0$-transcendant sur $\a,\cc,t'_a$ et sur
$\a,\ca,t'_a$ (et de même pour $t_3$ et $t_4$ sur les points
correspondants), et $t_1$ est $0$-transcendant sur $\a,\ab$.
\item $t_2$ est un uple vert $T$-générique sur $\a$.
\end{enumerate}
Ainsi, les uples $\a$, $\cc$, $t'_a$ et $t_2$, tout comme $\a$, $\ca$,
$t'_a$ et $t_2$, sont $0$-indépendants.\end{prop}
\begin{proof} Par le corollaire \ref{C:ind}, la structure $A_1\sse A_2\sse
A_3$ est autosuffisante et $t_1,t_2,t_3$ est une base linéaire de
$X=\U(A_1\sse A_2\sse A_3)$ sur 
$$Y=\U(\a)\cdot\U(\b)\cdot\U(\cc)\cdot\U(\ab)\cdot\U(\ca)\cdot\U(\cab).$$ 
Comme $A_1\sse A_2\sse A_3$ contient $\ab$, $\cc$ et $\cab$, elle
contient également $\sscl{\ab,\cc,\cab}$. En particulier,
$$t_4\in\U(A_4)=\U(\sscl{\ab,\cc,\cab})\subseteq A_1\sse A_2\sse A_3,$$
donc ils existent $t'_i\in\U(A_i)$ pour $1\le i\le3$, tels que
$t_4=t'_1\cdot t'_2\cdot t'_3$.
Par symétrie, pour tout $i<j\le3$, l'uple $t_i,t_j,t_4$ est aussi une
base linéaire de $X$ sur $Y$. Ceci implique 
que $\ld(t'_i/Y)=\ld(t_i/Y)$ pour chaque $1\le i\le 4$. \`A
transformation linéaire près de $t_1$, $t_2$ et $t_3$, on peut donc
supposer que $t_1 \cdot t_4 = t_2 \cdot t_3$, ce qui montre $(1)$.

Par le lemme \ref{L:preind}, pour tous
$i,j,k\in\{1,2,3,4\}$ distincts, l'uple $t_i$ est $0$-indépendant de
$t_j,t_k$ au-dessus de $\a,\b,\ab,\cab,\cc,\ca$.
Le lemme \ref{L:stab} appliqué à $t_1$, $t_2 \cdot t_3$ et $t_4$
donne que 
$$\tp_0(t_1/\acl_0(\a,\b,\ab,\cab,\cc,\ca))$$
est $0$-générique dans un translaté $V$ définissable
sur $\acl_0(\a,\b,\ab,\cab,\cc,\ca)$ d'un tore de
$(K^*)^{|t_1|}$. Puisque
$$t_1  \indi 0_{\a,\b,\ab} \cab,\cc,\ca,$$
le translaté $V$ est définissable sur $\acl_0(\a,\b,\ab)$. Or, pour
tout translaté d'un tore, sa dimension de Zariski
correspond à la dimension linéaire multiplicative du point générique.
Donc
$$\tr(t_1/\a,\b,\ab) =
\ld_\Q(t_1/\acl_0(\a,\b,\ab))=\dd(t_1/\acl_0(\a,\b,\ab))\le 0,$$
puisque $t_1\in A_1=\acl_0(\sscl{\a,\b,\ab})$. Ainsi $t_1$ est
$0$-algébrique sur $\a,\b,\ab$, et $A_1$ l'est aussi car
$$A_1=\acl_0(\sscl{\a,\b\,\ab})=\acl_0(\a,\b,\ab,t_1)=\acl_0(\a,\b,\ab),$$
ce qui montre $(2)$.

Pour $(3)$, posons $n = \tr(A_2/\a,\ca)$. Du fait que $\dd(A_2/\a\sse
\ca)=0$ et que $\a\sse \ca$ est autosuffisant, la base verte
$c_2,t_2,t_c$  de $A_2$ sur $\a\sse \ca$
est de longueur $2n$. Par $(2)$, on a 
$$A_2 =\acl_0(\cc,\a,\ca)= \acl_0(c_2,\a,\ca).$$ 
Puisque $\cc\indi0_{c_1}\a,\ca$, l'uple  $c_2$ reste
$0$-transcendant sur $\a,\ca$. Un calcul de pré\-di\-men\-sion donne que
$$0= \dd(c_2,t_2,t_c/\a,\ca)
= 2\tr(c_2/\a,\ca)-|c_2|-|t_2|-|t_c|.$$
Ainsi $|t_2|+|t_c| = |c_2|$. Comme $\cc$ est autosuffisant, on en déduit
$$\RM(\cc/c_1)=\dd(\cc/c_1)=2\tr(\cc/c_1)-\ld_\Q(\cc/c_1)
=2|c_2|-|c_2|-|t_c|=|t_2|.$$

Pour $(4)$, supposons  que  les $j-1$ premières 
coordonnées $t_{2,1},\ldots,t_{2,j-1}$ de l'uple $t_2$ sont
$0$-transcendantes sur $\a,\ca,t'_a$ et sur $\a,\cc,t'_a$. 
Rappelons que $c_2$ est $0$-transcendant sur $\a,\ca$ et que $|c_2|=|t_2|+|t_c|
= |t_2| +|t'_a|$. L'une des coordonnées $z$ de $c_2$ est donc 
$0$-transcendante sur $\a,\ca,t'_a,t_{2,1},\ldots,  t_{2,j-1}$. Si
$t_{2,j}$ est $0$-algébrique sur $\a,\ca,t'_a,t_{2,1},\ldots,
t_{2,j-1}$ alors $t_{2,j}z$ est lui $0$-transcendant. De plus, si
$t_{2,j}$ était $0$-transcendant sur 
$\a,\cc,t'_a,t_{2,1},\ldots, t_{2,j-1}$, l'élément $t_{2,j}z$ l'est
toujours.

Afin de conserver l'égalité $t_1\cdot  t_4 = t_2 \cdot t_3$, on multiplie
$t_{4,j}$ par $z$. Notons que si $t_{4,j}$ était $0$-transcendant sur
$\ab,\cc,t'_{ab},t_{4,1},\ldots,t_{4,j-1}$, alors $t_{4,j}z$ l'est toujours.
Enfin, si $t_{4,j}$ était $0$-transcendant sur $\ab,\cab,t'_{ab},t_{4,1},\ldots,
t_{4,j-1}$ mais $t_{4,j}z$ devient $0$-algébrique, alors $t_{4,j}z^2$ est à
nouveau $0$-transcendant, ainsi que $t_{2,j}z^2$ sur les points correspondants
($z^2$ dénote ici le carré de $z$ au sens du corps). 
Le résultat suit par récurrence et symétrie. 

Pour $(5)$, comme $a$ et $c$ sont indépendants, on a 
$$c_2\ind_{c_1} \a.$$
Par $(3)$, l'une des coordonnées $z$ de $c_2$ est
$T$-transcendante sur $\a,t_{2,1},\ldots,t_{2,j-1}$ pour $j \leq
|t_2|$. Si $t_{2,j}$ est algébrique  sur
$\a,t_{2,1},\ldots,t_{2,j-1}$,  l'élément
$t_{2,j}z^n$ est $T$-transcendant sur $\a,t_{2,1},\ldots,t_{2,j-1}$
pour tout $n>0$ 
(où $z^n$ dénote la puissance $n$-ième de $z$ au sens du corps). Comme $t_{2,j}$
est $0$-transcendant sur
$\a,\ca,t'_a,(t_{2,i}:i\not=j)$, il y a
au plus un $n$ pour lequel $t_{2,j}z^n$ est
$0$-algébrique sur $\a,\ca,t'_a,(t_{2,i}:i\not=j)$. De même pour
$t_{4,j}z^n$ au-dessus de $\ab,\cab,t'_{ab},(t_{4,i}:i\not=j)$. Il
existe donc un  $n$ commun permettant de préserver $(1)$ et
$(4)$. On conclut par récurrence.

La dernière remarque suit du lemme \ref{L:modify_t_a} et du point
$(4)$.\end{proof}

\begin{corollary}\label{C:alga2t2} La droite $A_2$ est $0$-algébrique
sur $\a,\cc,t'_a,t_2$. Les points $\cc$ et $\ca$ sont
$0$-interalgébriques sur $\a,t'_a,t_2$.\end{corollary}
\begin{proof} D'après la proposition \ref{modify}, la droite
$A_2$ est $0$-algébrique sur $\a,\cc,\ca$
et donc sur $\a,\cc,ca_2$, car $\ca =\acl_0(ca_1,ca_2)$ et $ca_1 \in
\acl_0(a_1,c_1)$. De plus, l'uple $t'_a,t_2$ est $0$-transcendant sur $\a,\cc$. 
\`A partir du lemme \ref{L:modify_t_a} et de la proposition
\ref{modify}, on obtient
$$\tr(A_2/\a,\cc)=\tr(ca_2/\a,\cc)=|ca_2|=|t'_a|+|t_2|=\tr(t'_a,t_2/\a
,\cc),$$
d'où $A_2 =\acl_0(\a,\cc,t'_a,t_2)$. En particulier, on a
$\ca\in \acl_0(\a,\cc,t'_a,t_2)$ et, par symétrie,
$\cc\in\acl_0(\a,\ca,t'_a,t_2)$.\end{proof}

On pose
$$\begin{aligned}\alpha&=\acl_0(\cc,\ca)\cap\acl_0(\a,t'_a,t_2),\\
\beta&=\acl_0(\ca,\cab)\cap\acl_0(\b,t'_b,t_3)\mbox{ et}\\
\gamma&=\acl_0(\cab,\cc)\cap\acl_0(\ab,t'_{ab},t_4).\end{aligned}$$
Remarquons que $a_1,t'_a\in\alpha$ et $b_1,t'_b\in\beta$, ainsi que
$ab_1,t'_{ab}\in\gamma$.

\begin{prop}\label{P:0gp}
Le diagramme
\vskip4mm
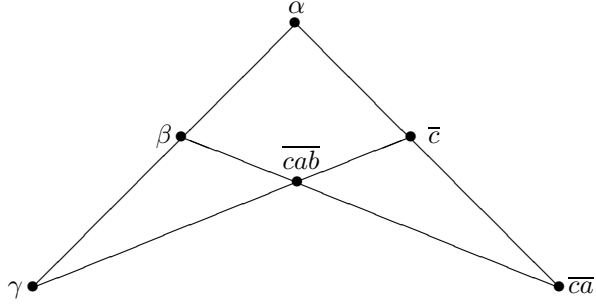
\begin{center}
\begin{picture}(200,100)
\put(0,0){\line(1,1){100}}
\put(200,0){\line(-1,1){100}}
\put(0,0){\line(5,2){143}}
\put(200,0){\line(-5,2){143}}
\put(-9,-2){$\gamma$}
\put(203,-2){$\ca$}
\put(97,104){$\alpha$}
\put(47,55){$\beta$}
\put(150,55){$\cc$}
\put(95,45){$\cab$}
\put(-2,-2){$\bullet$}
\put(197,-2){$\bullet$}
\put(97,98){$\bullet$}
\put(98,38){$\bullet$}
\put(54,55){$\bullet$}
\put(141,55){$\bullet$}
\end{picture}
\end{center}
est un quadrangle $0$-algébrique. De plus,
$$\cc,\ca \indi 0_{\alpha} \a,t'_a,t_2.$$
L'élément $\alpha$ et la $0$-base canonique $\cb_0(\cc,\ca/\a,t'_a,t_2)$ sont
donc $0$-interalgébriques.
\end{prop}

\begin{proof}  Montrons d'abord que  
\begin{equation}\tag{$\star$}
A_2\indi 0_\alpha\b,t'_b,t_3,\ab,t'_{ab},t_4.
\end{equation}
Le lemme \ref{L:preind} donne l'indépendance $A_3, A_4 \indi 0_{\cc,\ca}
A_2$. Celle-ci entraîne 
$$\b,t'_b,t_3,\ab,t'_{ab},t_4 \indi 0_{\cc,\ca} A_2$$ et donc
$\cb_0(\b,t'_b,t_3,\ab,t'_{ab},t_4 /A_2) \in \acl_0(\cc,\ca)$.

Il reste à montrer que $\cb_0(\b,t'_b,t_3,\ab,t'_{ab},t_4 /A_2) \in
\acl_0(\a,t'_a,t_2)$, c'est-à-dire
$$\b,t'_b,t_3,\ab,t'_{ab},t_4 \indi 0_{\a,t'_a,t_2} A_2.$$
Comme $A_3\indi 0_{\b} A_1$, on a $\ca,t'_b,t_3 \indi 0_{\b} A_1$.
Puisque
$\b$, $\ca$ et $t'_b,t_3$ sont $0$-indé\-pen\-dants par la
proposition \ref{modify},  il suit que $\ca,t'_b,t_3 
\indi 0 A_1$ et
\begin{equation}\tag{\dag}\ca \indi 0_{A_1} t'_b,t_3.\end{equation}
L'indépendance $A_3  \indi 0_{\b,\ca} A_1, A_2$ entraîne $t'_b,t_3 
\indi 0_{A_1,\ca}\cc,t'_a,t_2$~; cette indépendance et (\dag)
donnent par transitivité
$t'_b,t_3 \indi 0_{A_1} \cc, t'_a,t_2$ et donc
$$t'_b,t_3 \indi 0_{A_1,t'_a,t_2}\cc.$$ 
Comme $t_1\cdot t_4= t_2\cdot t_3$ et $t'_a\cdot t'_b=t'_{ab}$, on a
\begin{equation}\tag{\ddag}
t'_{ab},t_4 \indi 0_{A_1,t'_a,t_2}\cc.
\end{equation}
Ensuite $A_2\indi 0_{\a} A_1$ implique $\cc\indi 0_{\a,t'_a,t_2}
A_1$, ce qui avec (\ddag) entraîne 
$$\cc\indi 0_{\a,t'_a,t_2} A_1,t'_{ab},t_4.$$
Comme $A_2\subseteq\acl_0(\a,\cc,t'_a,t_2)$ d'après le corollaire
\ref{C:alga2t2},
$$A_2\indi 0_{\a,t'_a,t_2}A_1,t'_{ab},t_4.$$ 
\`A nouveau les identités $t_1\cdot t_4= t_2\cdot t_3$ et $t'_a\cdot
t'_b=t'_{ab}$ donnent 
$$\b,t'_b,t_3,\ab,t'_{ab},t_4 \indi 0_{\a,t'_a,t_2} A_2,$$
ce qui entraîne ($\star$).

Nous allons maintenant vérifier que les points et les droites du
diagramme satisfont les relations d'un quadrangle $0$-algébrique.

L'indépendance ($\star$) entraîne
$$\cc,\ca\indi 0_{\alpha}\b,t'_b,t_3,\ab,t'_{ab},t_4,$$
ce qui implique 
\begin{equation}\tag{$\sharp$}\cc,\ca\indi0_\alpha \beta,\gamma.\end{equation}
Alors, les droites $(\alpha, \beta, \gamma)$ et $(\alpha, \cc, \ca)$
sont $0$-indépendants sur leur intersection $\alpha$.
Par symétrie, chaque autre droite est $0$-indépendante de $(\alpha,
\beta,\gamma)$ sur leur intersection. La $0$-indépendance de deux
droites différentes de $(\alpha,\beta,\gamma)$ sur leur intersection
découle de l'indépendance des droites $A_i$.

Par la proposition \ref{modify}, $$\a,t'_a,t_2\indi0\cc \quad
\text{et} \quad \a,t'_a,t_2\indi0\ca.$$ Ainsi
$\alpha$, $\cc$ et $\ca$ sont deux à deux $0$-indépendants. D'après le
corollaire \ref{C:alga2t2} appliqué aux droites $A_3$ et
$A_4$, on a $\ca\in\acl_0(\b,\cab,t'_b,t_3)$ et
$\cab\in\acl_0(\ab,\cc,t'_{ab},t_4)$, ce qui implique
$\ca\in\acl_0(\b,t'_b,t_3,\ab,\cc,t'_{ab},t_4)$. Par ($\star$),
on conclut $$\ca\in\acl_0(\alpha,\cc).$$ 
On a de même $\cc \in\acl_0(\alpha,\ca)$. Donc chaque point de la
droite $(\alpha, \cc, \ca)$ est
$0$-al\-gé\-brique sur les deux autres. 

Par symétrie, on obtient les mêmes propriétés pour les droites
$(\beta, \cab, \ca)$ et $(\gamma, \cab, \cc)$.

Pour la droite $(\alpha, \beta,\gamma)$, l'indépendance 
$A_2\indi0_{\ca} A_3$ implique $\alpha\indi0_{\ca}\beta$. Comme
$\alpha\indi0\ca$ on a $\alpha\indi0\beta$. Puisque $\cab \in
\acl_0(\beta,\ca)$, l'indépendance $(\sharp)$ implique
$$\gamma\indi0_{\alpha,\beta}\cc,\ca,\cab,$$
et donc $\gamma\in\acl_0(\alpha,\beta)$. Par symétrie, les points
$\alpha$, $\beta$ et $\gamma$ sont deux-à-deux $0$-indépendants et
chacun est $0$-algébrique sur les deux autres.

Il s'agit bien d'un quadrangle $0$-algébrique.

Pour la dernière affirmation, notons que  l'uple $\alpha$ 
est dans $\acl_0(\a,t'_a,t_2)$, par définition. Or, la proposition
\ref{modify} entraîne  $\cc\indi0\a,t'_a,t_2,\alpha$, ce qui donne $\cc 
\indi 0_{\alpha} \a, t'_a,t_2$, d'où 
$$\cc,\ca \indi 0_{\alpha} \a,t'_a,t_2\,,$$
car $\ca\in\acl_0(\cc,\alpha)$.
\end{proof}

Le théorème de la configuration de groupe donne, sur des para\-mètres
$B=\acl(B)$ indépendants de $a,b,c$, un groupe $T_0$-$*$-définissable
$0$-connexe $H$, et des éléments
$0$-génériques $0$-indépendants $h$ et $h'$ dans $H$ sur $B$, tels que
$\alpha$ est $0$-interalgébrique avec $h$ sur $B$, ainsi que $\beta$
avec $h'$ et $\gamma$ avec $hh'$. Notons que les éléments $h$ et $h'$
ne sont pas forcément indépendants et que leur $T$-type n'est pas
générique dans $H$ sur $B$.

Comme $B$ est indépendant de $\acl(a,b,c)$, toutes les indépendances
et $0$-in\-dé\-pen\-dan\-ces obtenues à ce point sont préservées par
l'adjonction de $B$ à la base.

\begin{corollary}\label{C:a_et_alpha}
Les uples $h$ et $(t'_a,t_2)$ sont $0$-interalgébriques sur $\a, B$.
\end{corollary}
\begin{proof} Par l'interalgébricité de $\alpha$ et $h$ sur $B$, il
suffit de le démontrer pour $\alpha$ à la place de $h$. Notons
d'abord que l'uple $\alpha$ est dans $\acl_0(\a,t'_a,t_2)$ par
définition. Réciproquement, la proposition précédente donne 
$$\cc,\ca\indi 0_{\alpha}\a,t'_a,t_2,$$
d'où $\cc,\ca \indi 0_{\a,\alpha} t'_a,t_2.$
Comme $t'_a,t_2\in\acl_0(\a,\cc,\ca)$, on conclut que
$$t'_a,t_2\in\acl_0(\a,\alpha).\qedhere$$\end{proof}
Pour la suite, on pose $r=|t'_a|$ et $s=|t_2|$.
\begin{prop}\label{P:S_isog} 
Soit $S = \St_0(h,t'_a,t_2/\a,B)$ le $0$-stabilisateur de $(h,t'_a,t_2)$ 
dans $H \times(K^*)^{r+s}$ et $N$ sa projection sur $H$. 
Alors $S$ et $N$ sont $0$-connexes
et $T_0$-définissables sur $B$, et $(h,t'_a,t_2)$ est
$0$-générique sur $\a, B$ dans le translaté $S\cdot(h,t'_a,t_2)$, qui
est aussi $T_0$-définissable sur $\acl_0(\a,B)$. Le
sous-groupe $N$ est central dans $H$, et $S$ induit une isogénie entre
$N$ et $(K^*)^{r+s}$.

De plus, l'uple $a_1$ est $0$-algébrique sur $B$ et le paramètre canonique $\ulcorner
Nh\urcorner$ du translaté $Nh$.\end{prop}
\begin{proof} Puisque $h\in\acl_0(A_2,B)$, l'indépendance
$A_2\indi0_{\a,B}A_1$ implique 
\begin{equation}\tag{\dag}h,t'_a,t_2\indi0_{\a,B} \b,\ab\,,\end{equation}
et  $S=\St_0(h,t'_a,t_2/\a,\b,\ab,B)$.
 
Vérifions d'abord que 
$$(h,t'_a,t_2),\quad (h',t'_b,t_3)\quad\text{et}\quad (hh',t'_{ab},t_1\cdot 
t_4)=(h,t'_a,t_2)\cdot(h',t'_b,t_3)$$
sont deux-à-deux $0$-indépendants au-dessus de $\a,\b,\ab,B$.

L'indépendance $A_2 \indi 0_{\a,B} A_1$ donne
$t'_a,t_2\indi0_{\a,\ca,B}\b,\ab$. Comme $t'_a,t_2 \indi 0_{\a,B} \ca$ par la
proposition \ref{modify}, on a que 
$t'_a,t_2\indi0_{\a,B}\b,\ab,\ca$. Puisque $h$ est
$0$-algébrique sur $\a,t'_a,t_2,B$, on obtient 
$$h,t'_a,t_2\indi 0_{\a,\b,\ab,B}\ca.$$
Enfin, l'indépendance $A_1 A_2 \indi 0_{\b,\ca,B} A_3$ implique
$$h,t'_a,t_2\indi 0_{\a,\b,\ab,\ca,B}h',t'_b,t_3\,,$$
et donc par transitivité
$$h,t'_a,t_2\indi0_{\a,\b,\ab,B}h',t'_b,t_3.$$
Par symétrie, puisque $t_1\in\acl_0(\a,\b,\ab)$, on obtient également
$$hh',t'_{ab},t_1\cdot t_4  \indi 0_{\a,\b,\ab,B}
h',t'_b,t_3\quad\text{et}\quad hh',t'_{ab},t_1\cdot t_4 \indi
0_{\a,\b,\ab,B} h,t'_a,t_2.$$
D'après le lemme \ref{L:stab}, le groupe $S$ est $0$-connexe, et le point
$(h,t'_a,t_2)$ est $0$-générique sur $\acl_0(\a,\b,\ab,B)$ dans le 
translaté $S\cdot(h,t'_a,t_2)$, qui est
$T_0$-définissable sur $\acl_0(\a,\b,\ab,B)$. Notons que $N$ doit
être aussi $0$-connexe. 

Puisque $S=\St_0(h,t'_a,t_2/\a,B)$, il est $T_0$-définissable sur
$\acl_0(\a,B)$. Le translaté $S\cdot(h,t'_a,t_2)$ est
$T_0$-définissable sur 
$$\acl_0(h,t'_a,t_2,\a,B)\cap\acl_0(\a,\b,\ab,B)=\acl_0(\a,B),$$
par l'indépendance (\dag). En particulier, sa projection $Nh$ est
également $T_0$-définissable sur $\acl_0(\a,B)$, et $\ulcorner
Nh\urcorner\in\acl_0(\a,B)$.

D'après le lemme \ref{L:stab}, on a aussi 
$$S=\St_0(hh',t'_{ab},t_1\cdot t_4/\a,\b,\ab,B).$$
Or, l'indépendance $A_4\ind_{\ab,B} A_1$ et  $hh'\in\acl_0(A_4,B)$
impliquent
$$hh',t'_{ab},t_4\indi0_{\ab,t_1,B}\a,\b\,;$$
donc $S$ est $T_0$-définissable sur $\acl_0(\ab,t_1,B)$.
Comme $\a$, $\ab$ et $t_1$ sont $0$-indépendants sur $B$
d'après la proposition \ref{modify}, on a
$\acl_0(\a,B)\cap\acl_0(\ab,t_1,B)=B$. On conclut que $S$, ainsi que
$N$, sont $T_0$-définissables sur $B$.

Puisque $h$ est $0$-générique dans $H$ sur $B$, il est $0$-générique dans le
translaté $Nh$ sur $\acl_0(\ulcorner Nh\urcorner, B)$.
Notons que $$a_1\in\a\cap\alpha\subseteq\a\cap\acl_0(h,B).$$
Soit $(n,t',t)$ un $0$-générique de $S$ sur $B,\a,h$. Alors
$n\in\St_0(h/\a,B)$, d'où 
$$nh\equiv^0_{\acl_0(\a,B)}h$$
et $a_1\in\acl_0(nh,B)$.
Comme $n$ est $0$-générique dans $N$ sur $B,\a,h$, on a $nh\indi0_{\ulcorner
Nh\urcorner,B} h$. Alors 
$$a_1\in \acl_0(nh,B)\cap\acl_0(h,B) =\acl_0(\ulcorner
Nh\urcorner,B).$$
De plus, l'élément $h$ est $0$-générique dans $Nh$ sur
$\a,B$.

D'après la proposition \ref{modify}, la paire $(t'_a,t_2)$ est
$0$-générique sur $\a,B$.
La $0$-interal\-gé\-bricité entre $h$ et $(t'_a,t_2)$ sur $\a,B$ entraîne,
par le lemme \ref{F:iso}, que le stabilisateur $S$ est une
iso\-génie $T_0$-définissable sur $B$ entre $N$ et $(K^*)^{r+s}$. En
particulier, le sous-groupe $N$ est l'enveloppe $T_0$-définissable de
sa torsion, et sa $n$-torsion est finie pour chaque $n$, car c'est le cas de
$(K^*)^{r+s}$ et $N$ est $0$-connexe. Pour montrer que $N$ est
central dans $H$, il suffit donc de vérifier que $N$ est normal dans $H$~: 
ainsi sa $n$-torsion est $H$-invariante pour tout $n$~; comme elle est finie et
$H$ est connexe, elle est centrale.

On considère le $0$-stabilisateur $S_1=\St_0(h',t'_b,t_3/\a,\b,\ab,B)$ et sa
projection $N_1$ sur $H$. Par symétrie, les groupes $S_1$ et $N_1$ sont
$T_0$-définissables sur $\acl_0(\b,B)$. D'après le lemme \ref{L:stab} on a
$N_1=h^{-1}Nh$, donc $N_1$ est également $T_0$-définissable sur
$\acl_0(h,B)$. Comme $h\indi0_B\b$, le groupe $N_1$ est $T_0$-définissable sur
$B$, et $h^{-1}Nh$ ne dépend pas du $0$-générique $h$ de $H$ sur $B$. Si $h_1$
est un deuxième $0$-générique de $H$ sur $B$ indépendant de $h$, on a
$h_1^{-1}Nh_1=h^{-1}Nh$. Ainsi $hh_1^{-1}$ est un $0$-générique de $H$ sur $B$
qui normalise $N$~; par connexité tout $H$ normalise $N$.
\end{proof}

Le lemme \ref{L:stab} appliqué aux uples $(h,t'_a,t_2)$, $(h',t'_b,t_3\inv{t_1})$ et $(hh',t'_{ab},t_4)$ sur $\acl_0(\a,\b,\ab,B)$, qui contient $t_1$, entraîne également le résultat suivant~:
\begin{remark}\label{R:S_isog}
On a $$S = \St_0(h,t'_a,t_2/\a,B) = \St_0(hh',t'_{ab},t_4/\ab, B)$$
et le translaté $S\cdot(hh',t'_{ab},t_4)$ est $T_0$-définissable sur
$\acl_0(\ab, B)$.
\end{remark}

Avec les notations précédentes, l'uple $t_2$ est un uple vert générique 
de $(K^*)^s$ sur $\a$, par la proposition \ref{modify}. Or, quoique
 l'uple $(t_2,t'_a)$ est $0$-générique dans $(K^*)^{r+s}$ sur
$\a, B$, il ne l'est pas forcement au sens de $T$.
Cependant, nous allons modifier $h$ pour supprimer $t'_a$.

Par connexité de $K^*$, l'isogénie $S$ est surjective. Il y a donc
$h_0$ et $h'_0$ dans $N$ avec $(h_0, t'_a,\bar1)$ et
$(h'_0,t'_b,\bar1)$ dans $S$, et leur produit $(h_0h'_0,t'_{ab},\bar1)$ est
également dans $S$. Posons $k=h_0^{-1}h$ et $k'=h_0^{\prime-1}h'$.
\begin{lemma}\label{L:kinteralgt2} On a
$$\acl_0(\cc,\ca,B)\cap\acl_0(\a,t_2,B)=\acl_0(k,B)\quad\text{et}\quad
\cc,\ca\indi0_{k,B}\a,t_2\,.$$  
Les uples $k$ et $t_2$ sont $0$-interalgébriques sur
$\a,B$, et $a$ est algébrique sur~$k,B$.
\end{lemma}

\begin{proof} Le translaté $S\cdot(h,t'_a,t_2)=S\cdot(k,\bar 1,t_2)$ est
$T_0$-définissable sur $\acl_0(\a,B)$. Comme $S$ est une isogénie, il
suit que $k$ et $t_2$ sont $0$-interalgébriques sur $\a,B$. Puisque
$h_0\in\acl_0(t'_a,B)$ et $t'_a,h\in\acl_0(\cc,\ca,B)$, on a
$$k\in\acl_0(\cc,\ca,B)\cap\acl_0(\a,t_2,B).$$
Puisque $Nh=Nk$, la proposition \ref{P:S_isog} entraîne que 
$$a_1\in\acl_0(\ulcorner Nh\urcorner, B)=\acl_0(\ulcorner Nk\urcorner,
B)\subseteq\acl_0(k,B).$$
Le fait \ref{F:reduitgroupe} donne $\cc,\ca \indi 0_{a_1,B} \a$, donc 
$$ \cc,\ca \indi 0_{k,B} \a, t_2,$$
ce qui entraîne en particulier
$$\acl_0(\cc,\ca,B)\cap\acl_0(\a,t_2,B)=\acl_0(k,B).$$
Enfin, puisque $\cc\sse\ca\sse B$ est autosuffisant, la clôture
autosuffisante $\sscl{k,B}$ est dans $\acl_0(\cc,\ca,B)$. Ainsi,
$$\cc,\ca \indi 0_{\sscl{k,B}}a_2,t_a,t_2.$$
Comme $a_2,t_a,t_2$ est linéairement indépendant sur $\U(\cc\sse\ca\sse B)$, on
obtient
$$\dd(a_2,t_a,t_2/\sscl{k,B}) =\dd(a_2,t_a,t_2/\cc\sse\ca\sse B)
=\dd(A_2/\cc\sse\ca\sse B)=0,$$
ce qui permet de conclure  que $a_2$ est algébrique sur
$\sscl{k,B}$ et donc sur $k,B$. Comme
$a\in\acl(a_1,a_2)$, on conclut que $a$ est aussi
algébrique sur $k,B$.\end{proof}

On obtient ainsi le théorème A de l'introduction. 
\begin{theorem}\label{T:theoremeA} Un groupe interprétable dans un corps vert
collapsé est isogène à un quotient d'un sous-groupe définissable d'un
groupe algébrique par un sous-groupe central, qui est lui isogène à une
puissance du sous-groupe coloré $\U$.\end{theorem}
\begin{proof}Avec les définitions et notations précédentes, on considère le
groupe
$$\Gamma=\{n\in H:\exists\ t\in\U^s\ (n,\bar 1,t)\in S\}.$$ 
Notons que $\Gamma\le N$ est également central dans $H$, et
$\RM(\Gamma)=\RM(\U^s)=s$ comme $S$ est une isogénie.

Puisque $S$ induit une isogénie entre $N$ et $(K^*)^{r+s}$, l'intersection
$S\cap(H\times\{\bar 1\}\times \U^s)$ induit une isogénie entre $\Gamma$ et
$\U^s$. Comme $t_2$ est vert, la projection sur la première coordonnée de 
$$S\cdot(k,\bar 1,t_2)\cap (H\times\{\bar 1\}\times\U^s)$$
est $\Gamma k=k\Gamma$.  Puisque $S\cdot(k,\bar 1,t_2)=S\cdot(h,t'_a,t_2)$ est
définissable sur $\acl(a,B)$, la projection $k\Gamma$ l'est aussi~;
comme $t_2$ est un uple vert générique sur $a,B$, on a 
$\RM(k/a,B)=\RM(t_2/a,B)=s$ et $k$ est générique dans $k\Gamma$. De
même, le translaté $k'\Gamma$ est définissable sur $\acl(b,B)$, et $kk'\Gamma$
est définissable sur $\acl(ab,B)$, par la remarque \ref{R:S_isog}.

Les uples  $(a,k\Gamma)$, $(b,k'\Gamma)$ et 
$(ab,kk'\Gamma)=(a,k\Gamma)\cdot(b,k'\Gamma)$ sont deux à deux
indépendants sur $B$. Par le lemme \ref{L:stab}, le point $(a,k\Gamma)$ 
est générique dans un translaté, définissable sur $B$, de son 
stabilisateur $$\St(a,k\Gamma/B),$$ dans 
$G\times(H/\Gamma)$.

Le lemme \ref{F:iso} et la remarque \ref{R:iso} donnent une endogénie de
$G$ dans $H/\Gamma$, car $a$ est générique dans $G$ sur $B$. 

Pour vérifier que le noyau de cette endogénie est fini, il suffit de
montrer que $a$ et $k\Gamma$ sont interalgébriques sur $B$, c'est-à-dire   
$\RM(a/B)\le\RM(k\Gamma/B)$. D'après le lemme \ref{L:kinteralgt2}, on obtient
$$\RM(k/B)=\RM(k/a,B)+\RM(a/B)=\RM(t_2/a,B)+\RM(a/B)=s+\RM(a/B).$$
Comme $k$ est générique dans $k\Gamma$, on a par ailleurs $$\RM(k\Gamma/B)\ge\RM(k/B)-\RM(k/k\Gamma,B) = \RM(k/B)-s.$$
Ainsi 
$\RM(k\Gamma/B)\ge \RM(a/B)$, comme souhaité.

Par stabilité, le groupe $H$ est une limite projective $\varprojlim
\pi_i(H)$, où chaque  $\pi_i(H)$ est un groupe $T_0$-définissable et
donc algébrique. Par compacité il existe $i_0$ tel que $\U^s$ est isogène à $\pi_{i_0}(\Gamma)$, donc $G$ est isogène à son image
dans $\pi_{i_0}(H)/\pi_{i_0}(\Gamma)$.

\end{proof}

\section{Sous-groupes colorés}\label{S:ss}
Rappelons que tout groupe simple définissable dans un corps coloré se
plonge dans un groupe algébrique d'après \cite[Corollaire
5.10]{BPW09}. Ceci et le théorème A nous amènent à étudier les
sous-groupes définissables d'un groupe algébrique.

Comme dans les parties \ref{S:liebe} et \ref{cherchez}, on notera $T$ la théorie
d'un corps coloré $K$ et l'indice $0$ fera référence au réduit du pur
corps algébriquement clos. Ainsi un
groupe $T_0$-définissable est un groupe algébrique.

Nous commençons par une conséquence de la proposition \ref{P:expansion},
dans le cadre des corps colorés.

\begin{remark}\label{R:expansion}
À l'intérieur d'un groupe algébrique, on considère un translaté $C$ d'un
sous-groupe connexe définissable $S$, le tout défini sur un ensemble
al\-gé\-bri\-que\-ment clos $A$. Si le générique $b$ de $C$ vérifie $\U(\sscl{Ab})=\U(A)$, alors $S$ 
est un sous-groupe algébrique. 
\end{remark}

\begin{proof} 
Dans tout amalgame de Hrushovski, un type sur un ensemble algébriquement clos est stationnaire, puisque sa seule extension non-déviante correspond à l'amalgame libre.
Donc $p=\tp(b/A)$ est stationnaire. Par la proposition \ref{P:expansion}, il
suffit de montrer que $p$ est l'unique complétion de rang maximal du
$0$-type $p_0=\tp_0(b/A)$. 

Notons que $\sscl{Ab}$ coïncide avec la $0$-sous-structure
$B$ engendré par $A$ et $b$, puisqu'il n'y a pas de points colorés en 
dehors de $A$. 

Si $x\models p_0$, la $0$-sous-structure $X$ engendrée par $A$ et $x$
est $T_0$-isomorphe à $B$. Si $X=\sscl X$ et $\U(X)=\U(A)$,  alors
$x\models p$, car le diagramme d'une sous-structure autosuffisante détermine son type. Sinon, on a $\U(X)\supsetneq\U(A)$ ou $X\subsetneq\sscl X$. Dans ces deux cas,
$$\dd(\sscl X/A)\leq \dd(X/A)\leq \dd(B/A)=\dd(\sscl{Ab}/A),$$
avec au moins une inégalité stricte, ce qui implique $\RM(x/A)<\RM(b/A)$.
\end{proof}

\begin{theorem}\label{T:Bvert} Dans un corps vert, tout
sous-groupe définissable connexe $G$ d'un groupe algé\-brique a 
un sous groupe normal algé\-brique $N$, tel que le quotient $G/N$ est
définis\-sablement isomorphe 
à une puissance cartésienne du sous-groupe multiplicatif vert.
\end{theorem}

\begin{proof}
Dans un corps vert, considérons  un sous-groupe connexe définissable $G$ d'un groupe algébrique. On  suppose que $G$ est définissable sur
$\emptyset$.

Soient $a$ et $b$ deux génériques indépendants de $G$ et notons
$c=ab$ leur produit. Alors $\sscl a$ et $\sscl b$ sont en amalgame
libre, et la structure $\sscl a\sse \sscl b$ est autosuffisante. L'élément  
$c$ est $0$-algébrique sur cette structure, car la loi de groupe est 
$T_0$-définissable. Ainsi 
$$\sscl c\subseteq \acl_0(\sscl a\sse \sscl b),$$
par la remarque \ref{R:verts}(\ref{R:verts:pasdenouveauxverts}), qui
est également valable dans le cas non-collapsé.

En particulier, la base verte $t$ de $\sscl c$ est une combinaison
linéaire des bases vertes $r$ et $s$ de  $\sscl a$ et $\sscl b$,
respectivement. Par le même argument qu'au début de la preuve de la
proposition \ref{modify}, on peut supposer que $t=r\cdot s$.
Comme $r$, $s$ et $t$ sont deux à deux indépendants, 
ils sont également $0$-indépendants \cite[Lemme 2.1]{BPW09}. Par
le lemme \ref{L:stab}, l'uple $r$ est $0$-générique dans un translaté
$T_0$-définissable sur $\emptyset$
de son $0$-stabilisateur dans $(K^*)^{|r|}$, qui de plus
est $0$-connexe. Or, les seuls sous-groupes algébriques
$0$-connexes de $(K^*)^{|r|}$ sont des tores, 
dont la dimension de Zariski correspond à la dimension linéaire du
point générique. On conclut que 
$r$ est un uple $0$-transcendant. Comme
la structure engendrée par $r$ est autosuffisante dans $\sscl a$ et donc
autosuffisante, l'uple vert $r$ est générique dans $\U^{|r|}$.

Puisque $(a,r)$, $(b,s)$ et
$(c,t)$ sont deux à deux indépendants et que $r$ est algébrique sur
$a$, les lemmes  \ref{L:stab}, \ref{F:iso} et la remarque
\ref{R:iso} entraînent que le stabilisateur $\St(a,r)$
induit une endogénie définissable $\phi$ de $G$ sur $(\U^*)^{|r|}$.
Le conoyau de cette endogénie est trivial, car $\U$ n'a pas de torsion et 
donc pas de sous-groupes finis non-triviaux. 

Pour montrer que le noyau $\ker(\phi)$ est un groupe algébrique, il suffit 
de le faire pour sa composante connexe $N$, car nous sommes à 
l'intérieur d'un groupe algébrique. Le paramètre canonique de $\ker(\phi)a$
est interdéfinissable avec $\phi(a)=r$, et donc le paramètre canonique de
$Na$ est algébrique sur $r$. Si $n$ est générique dans $N$ sur $a$, alors 
il l'est aussi sur $r$, donc $na$ est un générique de $Na$ sur $\acl(r)$. La
remarque \ref{R:expansion} permet de conclure si l'on vérifie que 
$$\U(\sscl{na,\acl(r)})=\U(\acl(r)).$$
Puisque  $(n,\bar1)\in\St(a,r)$, les points $a$ et
$na$ ont même type sur $\acl(r)$, ce qui donne l'égalité
$\U(\sscl{na,\acl(r)})=\U(\sscl{a,\acl(r)})$.

Posons $C=\acl(r)\cap \sscl{a}$. Alors $a$ et $\acl(r)$ sont trivialement 
indépendants au-dessus de $C$, car $r\in C$. La caractérisation
de l'indépendance nous donne que $\sscl{a,\acl(r)}$ est l'amalgame libre de 
$\sscl{a}$ et $\acl(r)$ au-dessus de $C$. De plus, puisque
$\U(\sscl{a})=\sscl r$ est le sous-groupe engendré par $r$, on obtient
$$\U(\sscl{a,\acl(r)})=\U(\acl(r)).\qedhere$$\end{proof}

Dans le cas des corps rouges, les sous-groupes algébriques  de $(K^+)^{|r|}$ sont donnés par des systèmes de $p$-polynômes, dont la
dimension de Zariski ne correspond pas forcement à la dimension linéaire sur le
corps premier. De plus, il y a des sous-groupes rouges finis. Néanmoins, la
preuve précédente s'adapte, nous permettant d'obtenir le résultat suivant~:

\begin{theorem}\label{T:Brouge}
Dans un corps rouge, tout sous-groupe définissable connexe
$G$ d'un groupe algébrique a un sous groupe normal algé\-brique $N$ tel que le
quotient $G/N$ est définis\-sa\-blement isogène au groupe des
points rouges d'un sous-groupe algébrique de $(K^+)^n$.
\end{theorem}

\begin{corollary}\label{C:Bsimple} Tout groupe simple définissable dans un corps coloré est définis\-sa\-blement isomorphe à un groupe algébrique. En particulier, aucun mauvais groupe n'est définissable dans un corps coloré.\end{corollary}
\begin{proof} On peut supposer que le groupe simple $G$ est infini. D'après
\cite[Théorème 6.4 et Corollaire 5.10]{BPW09}, le groupe $G$  est
définissablement isomorphe à un groupe linéaire. Dans le cas noir, il est
algébrique \cite[Proposition 2.4 et
conclusion p.1354]{Po99}. Dans les cas rouge et vert, d'après les
théorèmes \ref{T:Bvert} et \ref{T:Brouge}, il existe un sous-groupe algébrique
normal $N$ de $G$, tel que le quotient est isogène
à un groupe abélien. Par simplicité, le groupe $G=N$ est algébrique. 
\end{proof}

\section{La fusion au dessus de l'égalité}\label{S:fusion}

Une théorie $\omega$-stable a la propriété de la multiplicité
définissable (DMP) si, pour toute formule $\varphi(x,y)$, tout ordinal 
$\alpha$ et entier $n<\omega$, l'ensemble des paramètres $a$ tels que 
$\varphi(x,a)$ a rang de Morley $\alpha$ et multiplicité $n$ est 
définissable. 

Dans cette partie, nous décrirons complètement les groupes définissables 
dans une fusion $T$ (libre ou collapsée) de deux théories fortement
minimales $T_1$ et $T_2$ avec la DMP, à langages disjoints
\cite{Hr92,BMPZ07}. Comme dans les parties précédentes, les notions 
modèle-théoriques seront prises au sens de $T$~; l'indice $i$,  pour $i=1,2$,
fera référence à la théorie $T_i$.

Rappelons que la fusion $T$ est obtenue en utilisant la prédimension 
$$\dd(X) = \dim_1(X) + \dim_2(X) -|X|,$$
où $\dim_1$ et $\dim_2$ sont les rangs de Morley respectifs de $T_1$ et $T_2$.
Comme dans la partie \ref{S:liebe}, cette prédimension induit un
opérateur, la clôture autosuffisante $\sscl.$, ainsi qu'une dimension,
notée $\dim$~: la prédimension de la clôture autosuffisante. Alors
$\bar a\in\acl(B)$ implique $\dim(\bar a/B)=0$~; la réciproque est
vraie dans le cas collapsé. En particulier, deux
uples interalgébriques ont même dimension. 

Pour deux uples $\bar a$, $\bar b$ et un ensemble $C$ tel que
$\sscl{aC}\cap\sscl{bC}=C$, l'indépendance
est caractérisée de la manière suivante~:
$$\bar a\ind_C \bar b\quad\text{si et seulement si}\quad  
\left\{\begin{array}{l}
\displaystyle\sscl{\bar a,C}\indi i_C \sscl{\bar b,C}\text{ pour
}i=1,\;2\,,\text{ et}\\
\sscl{\bar a,\bar b,C}=\sscl{\bar a,C}\cup\sscl{\bar
b,C}.\end{array}\right. 
$$ 

Après adjonction de constantes au langage, on peut supposer que chaque
théorie fortement minimale $T_i$ a élimination faible des
imaginaires. Nous pouvons ainsi parler de la dimension d'uples
d'imaginaires des théories $T_1$ et $T_2$.

\'Etant données deux sortes imaginaires $S_1$ de $T_1$ et $S_2$ de $T_2$, on entend par un ensemble interprétable dans  $S_1\times S_2$ la
projection d'un ensemble définissable réel. Nous
commençons par étudier les sous-groupes interprétables dans un
produit de groupes $T_i$-interprétables. La démonstration du lemme
suivant s'inspire de la preuve de la platitude dans
l'amalgame \emph{ab initio} \cite{Hr93}.

\begin{lemma}\label{L:fusion} Soient $H_i$ des groupes
$T_i$-interprétables pour $i=1,2$. Tout sous-groupe interprétable
connexe $K$ de $H_1 \times H_2$ est de la forme $K_1 \times K_2$, où
$K_i \leq H_i$ est $T_i$-interprétable. Un générique $g=(g^1,g^2)$ de
$K$ consiste en un couple indé\-pendant de génériques de chaque
$K_i$. De plus, au-dessus de la clôture algébrique des paramètres nécessaires, on a 
$\dim(g^i) = \dim_i(g^i)$ et $\dim(K) = \dim_1(K_1)+\dim_2(K_2)$.
\end{lemma}

\begin{proof} On peut supposer que le langage est relationnel et que
les groupes $H_1$, $H_2$ et $K$ sont interprétables sur $\emptyset$.
Nous allons vérifier d'abord que $g^1$ et $g^2$ sont $T$-indépendants pour
un générique $g=(g^1,g^2)$ de $K$.  Par élimination faible des imaginaires, fixons un uple réel fini $a=(a^1,a^2)$ tel 
que $a^i$ est $i$-algébrique sur l'imaginaire $g^i$, qui est lui 
$T_i$-définissable sur $a^i$. Posons $X=\sscl a\cap\acl(\emptyset)$, qui 
est autosuffisant comme intersection de deux ensembles autosuffisants.

Comme $X\subset \acl(\emptyset)$, il suffit de montrer que 
$a^1\ind_X a^2$, ce qui par la caractérisation 
précédente équivaut à montrer que les clôtures $\sscl{X,a^1}$ et 
$\sscl{X,a^2}$ sont $i$-indépendants pour chaque $i=1,2$ et que 
$\sscl{X,a}$ est la réunion de ces deux clôtures. 

Considérons une suite de Morley $(g_i)_{i <\omega}$ du type $\tp(g)$. 
Pour $i\neq j$, l'élément $g_i{g_j}^{-1}$ est aussi générique avec même 
type que $g$ sur $\acl(\emptyset)$, par connexité de $K$. On choisit alors 
$b_{i,j}=(b_{i,j}^1,b_{i,j}^2)$ tel que
$(g_ig_j^{-1},b_{i,j})$ a même type que $(g,a)$ sur
$\acl(\emptyset)$. Comme $g_ig_j^{-1}$ et $g_{k}g_{l}^{-1}$ sont
indépendants pour $(i,j)\not=(k,l)$, on a
$$\sscl{b_{i,j}}\cap\sscl{b_{k,l}} = \sscl{b_{i,j}}\cap
\acl(\emptyset) = \sscl{a}\cap \acl(\emptyset) =X.$$
Fixons $n>0$. Puisque $b_{i,j}^1\in\acl_1(b_{0,i}^1,b_{0,j}^1)$ pour
$i,j>0$, on obtient
$$\dim_1((b_{i,j}^1)_{i<j<n}/X)\le\dim_1((b_{0,k}^1)_{k<n}/X)\le
n\,\dim_1(a^1/X).$$

Si $$Y = \bigcup\limits_{i<j<n}\sscl{b_{i,j}},$$ 
alors
$$\begin{aligned}\dim_1(Y/X)&=
\dim_1((b_{i,j}^1)_{i<j<n}/X) + \dim_1(Y /X,(b_{i,j}^1)_{i<j<n})\\
&\leq n\,\dim_1(a^1/X) + \sum_{i<j<n}\dim_1(\sscl{b_{i,j}}/X,b_{i,j}^1)\\
&\leq n\,\dim_1(a^1/X) + \binom{n}{2}\dim_1(\sscl{a}/X,a^1).\end{aligned}$$
De même,
$$\dim_2(Y/X)\leq n\,\dim_2(a^2/X)+\binom{n}{2}\dim_2(\sscl{a}/X,a^2).$$
Puisque $Y$ est la réunion d'ensembles deux à deux disjoints au-dessus de 
$X$, qui est autosuffisant, on
obtient~:
$$\begin{aligned}
0 \leq\ \dd(Y/X) & = \dim_1(Y/X)+\dim_2(Y/X) - |Y\setminus X|\\
&\leq n\,\big(\dim_1(a^1/X) +\dim_2(a^2/X)\big)\\ 
&\quad+\binom{n}{2}\big(\dim_1(\sscl{a}/X,a^1)+\dim_2(\sscl{a}/X,a^2)-
|\sscl{a}\setminus X|\big).
\end{aligned}$$
En faisant tendre $n$ vers l'infini, on obtient   
\begin{equation*}\tag{$\star$}\begin{aligned}
0&\le\dim_1(\sscl{a}/X,a^1)+\dim_2(\sscl{a}/X,a^2)-|\sscl{a}\setminus X|\\ 
&=\dd(\sscl{a}/X,a)+\dim_1(a/X,a^1)+\dim_2(a/X,a^2)-|a\setminus
X|.\end{aligned}\end{equation*}
Notons que $$\dim_1(a/X,a^1)\le|a^2\setminus (X\cup a^1)|,$$
avec égalité si et seulement si $a^{2}\setminus (X\cup a^1)$
est un uple de points $1$-indépendants sur $X,a^1$. De même pour
$\dim_2(a/X,a^2)$. Comme
$$|a\setminus X|=|a^1\setminus (X\cup a^2)|+|a^2\setminus (X\cup
a^1)|+|(a^1\cap a^2)\setminus X|,$$
on en déduit que 
$$\dim_1(a/X,a^1)+\dim_2(a/X,a^2)-|a\setminus X|\leq -|(a^1\cap
a^2)\setminus X|.$$
La prédimension $\dd(\sscl{a}/X,a) \leq 0$. Donc, par ($\star$), on a 
$$\begin{aligned}
(a^1\cap a^2)&\subseteq X,\\
\dd(\sscl{a}/X,a) &= 0, \\ 
\dim_1(a/X,a^1) &=|a^2\setminus X|,\quad\text{et}\\
\dim_2(a/X,a^2) &=|a^1\setminus X|.
\end{aligned}$$
Ceci entraîne que $a^1$ et $a^2$ sont $i$-indépendants pour chaque
$i$ au-dessus de $X$. De plus, comme l'uple $a^2\setminus X$ est un
uple de points $1$-indépendants sur $X\cup a^1$, il suit que $X\cup
a^1$ est autosuffisant dans  $X\cup a$. Puisque $X\subseteq\sscl a$ et
$\delta(\sscl a/X,a)=0$, on a que $\sscl a=X\cup a$, qui est
autosuffisant. Donc
$X\cup a^i$ est aussi autosuffisant pour $i=1,\,2$, ce qui entraîne que 
$g^1$ et $g^2$ sont indépendants. 

Soit maintenant $K_i$ la projection de $K$ sur $H_i$. Le groupe $K_i$ est 
aussi connexe et interprétable sur $\emptyset$. De plus, l'élément $g^i$ 
est générique dans $K_i$.
Puisque $g^1\ind g^2$, l'uple $g=(g^1,g^2)$ est alors générique
dans $K_1\times K_2$. Ainsi, l'indice de $K$ dans $K_1\times K_2$ est
fini. Or, la connexité de chaque $K_i$ donne celle de $K_1\times
K_2$, donc $K=K_1\times K_2$.
 
Vérifions maintenant que $\dim(g^1)=\dim_1(g^1/\acl(\emptyset)$~:  
comme $X\cup a^1$ est autosuffisant, 
$$(X\cup a^1)\cap\acl(\emptyset)=X\qquad\text{et}\qquad
a^1\ind_X\acl(\emptyset),$$ 
on obtient que $\acl(\emptyset)\cup a^1$ est
autosuffisant, et $a^1\indi2_X\acl(\emptyset)$. Ainsi
$a^{1}\setminus \acl(\emptyset)$ reste $2$-indépendant sur
$\acl(\emptyset)$, et
$\dim(a^1)=\dim_1(a^1/X)=\dim_1(a^1/\acl(\emptyset))$.

Pour terminer, voyons que $K_1$ est $T_1$-interprétable. Puisque dans ce contexte, tout type sur un ensemble algébriquement clos est stationnaire,  il suffit de montrer, par la proposition \ref{P:expansion}, que le type
$\tp(g^1/\acl(\emptyset))$ est l'unique complétion de rang de Morley
maximal du $1$-type $\tp_1(g^1/\acl(\emptyset))$. Prenons
$(h,b)$ une réalisation de $\tp_1(g^1,a^1/\acl(\emptyset))$. 
Si $b\setminus \acl(\emptyset)$ n'est pas $2$-indépendant sur 
$\acl(\emptyset)$, alors sa prédimension chute et le rang de Morley de $h$ est
strictement inférieur à celui de $g^1$. Sinon, les ensembles
autosuffisants  $\acl(\emptyset)\cup a^1 $ et  $\acl(\emptyset)\cup b$ 
sont isomorphes, et donc $b$ et $a$ ont le même
type. Puisque  $g^1$ est $T_1$-définissable sur $a^1$, il suit que
$h$ et $g^1$ ont également même type.\end{proof}

\begin{remark}\label{R:fusion}
En particulier, tout sous-groupe interprétable connexe d'un groupe
$T_i$-interprétable est également $T_i$-interprétable, et sa dimension 
est égale à sa $i$-dimension.
Notons que l'égalité entre ces deux dimensions pour des ensembles
$T_i$-interprétables apparaît dans \cite[Theorem 2$(i)$]{Hr92}.
\end{remark}

La condition supplémentaire du théorème \ref{T:morphisme}
est donc vérifiée dans ce cadre~:
\begin{lemma}\label{L:supp} Si $G$ est un groupe connexe définissable, alors
pour tout homomorphisme définissable $\psi:G\to H$ sur un ensemble $A$
algébriquement clos vers un groupe $T_i$-interprétable, on a
$\dim_i(\psi(a)/A)\le\dim(a/A)$, pour $a$ un générique de $G$ sur $A$.
En particulier, il existe un tel homomorphisme avec $\dim_i(\psi(a)/A)$
maximal et fini.
\end{lemma}
\begin{proof}
Notons que $\psi(G)$ est un sous-groupe interprétable connexe de $H$, 
et est donc $T_i$-interprétable, par la remarque précédente. Si $a$ est générique dans $G$ sur $A$, alors $\psi(a)$ est générique dans
$\psi(G)$ sur $A$, et
$$\dim_i(\psi(a)/A)=\dim(\psi(a)/A)\le\dim(a/A)\qedhere$$
\end{proof}

La proposition suivante est évidente dans le cas d'une fusion collapsée.
\begin{prop}\label{P:dim0} Dans une fusion libre de
deux théories fortement minimales avec la DMP au dessus de l'égalité,
tout groupe définissable de dimension $0$ est fini.\end{prop} 
\begin{proof}  Soit $G$ un groupe connexe définissable sur
$\emptyset$ de dimension $0$. Sur un ensemble algébriquement clos de paramètres
indépendants que l'on ajoute au
langage, le théorème \ref{T:morphisme} et le lemme \ref{L:supp} donnent pour $k=1,2$ des 
morphismes interprétables $\phi_k:G\to H_k$ où $H_k$ est $T_k$-interprétable, 
tels que pour deux génériques indépendants $a$ et $b$ on a 
\begin{equation}\tag{\dag}\acl(b),\acl(ab)\indi k_{\phi_k(a)}\acl(a).
\end{equation}
Par la remarque \ref{R:fusion}, on peut supposer chaque morphisme surjectif en remplaçant $H_k$ par $\phi_k(G)$.
De plus, comme $G$ est de dimension nulle, chaque $H_k$ l'est aussi,
donc $\phi_k(a)$ est $k$-algébrique. L'indépendance $(\dag)$ donne ainsi
$$\sscl{b},\sscl{ab}\indi k\sscl{a}.$$ 
Comme $a$, $b$ et $ab$ sont deux à deux indépendants, on a
$$(\sscl{b}\cup\sscl{ab})\cap\sscl{a}
=(\sscl{b}\cap\sscl{a})\cup(\sscl{ab}\cap\sscl{a})=\acl(\emptyset).$$
Par sous-modularité,
$$\dd(\sscl{a}/\sscl{b}\cup\sscl{ab}) 
\le \dd(\sscl{a}/\acl(\emptyset))=\dim(a/\acl(\emptyset))=\dim(a)=0.$$
Ainsi, puisque  $\sscl{b,ab}=\sscl{b}\cup\sscl{ab}$ est autosuffisant, l'ensemble
$$\sscl{b}\cup\sscl{ab}\cup\sscl{a}$$
l'est aussi. La caractérisation de l'indé\-pen\-dance implique que
$$ b, ab\ind a.$$
Donc $a$ est algébrique, ce qui montre le résultat.
\end{proof}

Grâce aux résultats précédents, nous sommes maintenant en mesure de
montrer le théorème C de l'introduction~: les groupes définissables dans la fusion sur l'égalité (libre ou collapsée) sont essentiellement des produits de groupes interprétables dans chacune des théories de base.

\begin{theorem}\label{T:fusion} Tout groupe connexe définissable dans une fusion (libre ou collapsée) au-dessus de l'égalité de deux théories fortement minimales
avec la DMP est, modulo un noyau fini, isomorphe à un produit de
groupes interprétables dans chacune des théories. 
\end{theorem}

\begin{proof} Soit $G$ un groupe infini connexe $T$-définissable sur $\emptyset$. 
Sur un ensemble algébriquement clos de paramètres indépendants que l'on ajoute
au langage, le théorème \ref{T:morphisme}, la remarque \ref{R:fusion} et le lemme \ref{L:supp} donnent
des morphismes surjectifs interprétables $\phi_k:G\to H_k$ où $H_k$ est
$T_k$-interprétable, tels que pour deux génériques
indépendants $a$ et $b$ on a 
\begin{equation*}\tag{\dag}\acl(a),\acl(b)\indi k_{\phi_k(ab)}\acl(ab).
\end{equation*}
On pose $\phi=(\phi_1,\phi_2):G\to H_1\times H_2$.
Comme $\phi_1$ et $\phi_2$ sont surjectifs et $\phi(G)$ est connexe, le
lemme \ref{L:fusion} donne que $\phi$ est également surjectif. De plus,

\begin{equation*}\tag{\ddag}
\dim_1(\phi_1(a))+\dim_2(\phi_2(a))=\dim(\phi(a)) \leq \dim(a).
\end{equation*}
Considérons un troisième élément générique $c$ indépendant de $a,b$, et  le diagramme suivant~:
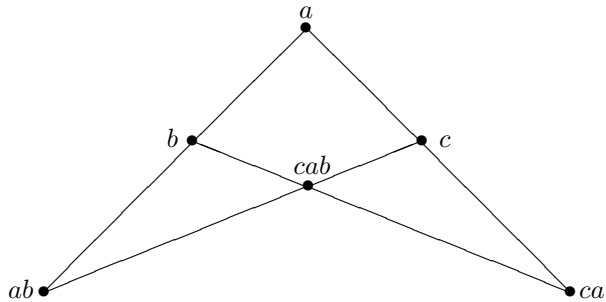
\begin{center}
\begin{picture}(200,110)
\put(0,0){\line(1,1){100}}
\put(200,0){\line(-1,1){100}}
\put(0,0){\line(5,2){143}}
\put(200,0){\line(-5,2){143}}
\put(-13,-2){$ab$}
\put(203,-2){$ca$}
\put(97,104){$a$}
\put(47,55){$b$}
\put(150,55){$c$}
\put(95,45){$cab$}
\put(-2,-2){$\bullet$}
\put(197,-2){$\bullet$}
\put(97,98){$\bullet$}
\put(98,38){$\bullet$}
\put(54,55){$\bullet$}
\put(141,55){$\bullet$}
\end{picture}\end{center}
De manière analogue au \cite[Lemma 15]{Hr93}, on considère des ensembles
autosuffisants finiment engendrés~:
$$\begin{aligned}
A_1&=\sscl{a,\phi(a),b,\phi(b),ab,\phi(ab)}\subseteq\acl(a,b,ab),\\
A_2&=\sscl{a,\phi(a),c,\phi(c),ca,\phi(ca)}\subseteq\acl(a,c,ca),\\
A_3&=\sscl{ca,\phi(ca),cab,\phi(cab),b,\phi(b)}\subseteq\acl(ca,cab,c),\\
A_4&=\sscl{ab,\phi(ab),c,\phi(c),cab,\phi(cab)}\subseteq\acl(ab,c,cab),\\
A_{\emptyset}&=\bigcup_{i=1}^4 A_i,\quad\text{et}\quad
A_s=\bigcap_{i\in s}A_i\quad\text{pour $s\subseteq\{1,2,3,4\}$.}\end{aligned}$$
Alors les $A_s$ sont des extensions finies de $\acl(\emptyset)$, et
$$\begin{array}{ll}
\sscl{a,\phi(a)}\subseteq A_{12}\subseteq\acl(a),&\sscl{b,\phi(b)}\subseteq A_{13}\subseteq\acl(b),\\
\sscl{ab,\phi(ab)}\subseteq A_{14}\subseteq\acl(ab),&\sscl{ca,\phi(ca)}\subseteq A_{23}\subseteq\acl(ca),\\
\sscl{c,\phi(c)}\subseteq A_{24}\subseteq\acl(c),&\sscl{cab,\phi(cab)}\subseteq A_{34}\subseteq\acl(cab).
\end{array}$$
Notons que $A_s=\acl(\emptyset)$ pour $|s|\ge3$, et
$$\begin{array}{ll}\dd(A_\emptyset)\ge\dim(A_{\emptyset})=3\,\dim(a),&\\
\dd(A_i)=\dim(A_i)=2\,\dim(a),&\text{pour }i\in\{1,2,3,4\}\\
\dd(A_{ij})=\dim(A_{ij})=\dim(a)&\text{pour }i\not=j,\\
\dd(A_s)=\dim(A_s)=0&\text{pour }|s|\ge3.\end{array}$$
Remarquons que :
\begin{equation*}\tag{$\star$}\begin{aligned}\dim(a) &=3\,\dim(a)-4\cdot
2\,\dim(a)+6\,\dim(a)\\
&=\sum_{s\subseteq \{1,2,3,4\} } (-1)^{|s|} \dim(A_s)
\leq \sum_{s\subset \{1,2,3,4\} } (-1)^{|s|} \dd(A_s)  \\
&=\sum_{\substack{s\subseteq
\{1,2,3,4\} \\ |s|\leq 2 }} (-1)^{|s|}\dim_1(A_s) +
\sum_{\substack{s\subseteq
\{1,2,3,4\} \\ |s|\leq 2 }} (-1)^{|s|}\dim_2(A_s),\end{aligned}
\end{equation*}
où  l'égalité finale provient de la modularité de la cardinalité.

Pour $k=1\,,\,2$, on obtient par sous-modularité :
$$\begin{aligned}\dim_k(A_{\emptyset}) &=\dim_k(A_1)+ \dim_k (A_2 A_3 A_4/A_1)\\
&\leq\dim_k(A_1)+\dim_k (A_2 A_3 A_4/A_{12} A_{13} A_{14})\\
&=\dim_k(A_1)+\dim_k(A_2A_3A_4)-\dim_k(A_{12} A_{13} A_{14})\\
&=\dim_k(A_1)+\dim_k(A_2)+\dim_k(A_3A_4/A_2)-\dim_k(A_{12} A_{13}/A_{14})\\
&\quad-\dim_k(A_{14})\\
&\le\dim_k(A_1)+\dim_k(A_2)+\dim_k(A_3A_4/A_{23}A_{24})-\dim_k(A_{12}A_{13}/A_{
14})\\
&\quad-\dim_k(A_{14})\\
&=\dim_k(A_1)+\dim_k(A_2)+\dim_k(A_3A_4)-\dim_k(A_{23}A_{24})\\
&\quad-\dim_k(A_{12} A_{13}/A_{14})-\dim_k(A_{14})\\
&=\dim_k(A_1)+\dim_k(A_2)+\dim_k(A_3)+\dim_k(A_4/A_3)-\dim_k(A_{23}A_{24})\\
&\quad-\dim_k(A_{12} A_{13}/A_{14})-\dim_k(A_{14})\\
&\le\dim_k(A_1)+\dim_k(A_2)+\dim_k(A_3)+\dim_k(A_4/A_{34})-\dim_k(A_{23}A_{24}
)\\
&\quad-\dim_k(A_{12} A_{13}/A_{14})-\dim_k(A_{14})\\
&=\dim_k(A_1)+\dim_k(A_2)+\dim_k(A_3)+\dim_k(A_4)-\dim_k(A_{34})\\
&\quad-\dim_k(A_{23}A_{24})-\dim_k(A_{12}
  A_{13}/A_{14})-\dim_k(A_{14}).\end{aligned}$$
Comme $A_{23}\ind A_{24}$ implique $A_{23}\indi k A_{24}$, nous avons 
$$\begin{aligned}\dim_k(A_{\emptyset})
  &\le\dim_k(A_1)+\dim_k(A_2)+\dim_k(A_3)+\dim_k(A_4)-\dim_k(A_{34})
  \\ &\quad-\dim_k(A_{14})-\dim_k(A_{23})-\dim_k(A_{24})-\dim_k(A_{12}
  A_{13}/A_{14}).\end{aligned}$$
En développant les séries alternées dans $(\star)$, on obtient 
$$\begin{aligned}\dim(a)&\le\sum_{k=1}^2\big(\dim_k(A_{12})+\dim_k(A_{13})
-\dim_k(A_{12}A_{13}/A_{14})\big)\\
&=\Delta_1 +\Delta_2,\end{aligned}$$
où $\Delta_k= \dim_k (A_{12} A_{13}) -\dim_k(A_{12} A_{13}/A_{14})$.
Puisque  $A_{12}\subseteq\acl(a)$, $A_{13}\subseteq\acl(b)$ et
$A_{14}\subseteq\acl(ab)$, l'indépendance $(\dag)$ donne 
$$A_{12}A_{13}\indi k_{\phi_k(ab)}A_{14}.$$
Ainsi,
$$\begin{aligned}\Delta_k&=\dim_k (A_{12} A_{13}) -\dim_k(A_{12}
A_{13}/A_{14})\\
&=\dim_k (A_{12} A_{13}) -\dim_k(A_{12} A_{13}/\phi_k(ab))\\
&=\dim_k(\phi_k(ab))-\dim_k(\phi_k(ab)/A_{12}A_{13})\\
&=\dim_k(\phi_k(ab)),\end{aligned}$$ 
car $\phi_k(ab) \in A_{14}$ est $k$-algébrique sur $\phi_k(a) \in A_{12}$ et $\phi_k(b) \in A_{13}$. D'où, $$\dim(a) \leq \Delta_1 +\Delta_2 =
\dim_1(\phi_1(ab)) +\dim_2(\phi_2(ab)) = \dim_1(\phi_1(a)) +\dim_2(\phi_2(a)).$$
Ceci entraîne par l'inégalité $(\ddag)$ que 
$$\dim(a) =\dim(\phi(a)).$$
Comme $\phi(a)$ est algébrique sur $a$, on conclut que
$\dim(a/\phi(a))=0$.

Or, l'élément $a$ est générique dans $a\cdot\ker\phi$ sur son
paramètre canonique, qui est $\phi(a)$. Comme la dimension ne
dépend pas du translaté, le noyau de $\phi$ a dimension $0$. Il est
donc fini d'après la proposition \ref{P:dim0}.\end{proof}

\begin{question} 
Peut-on aussi décrire les groupes interprétables dans la fusion
au-dessus de l'égalité~?

Existe-t-il une caractérisation analogue pour les groupes
définissables dans la fusion fortement minimale au-dessus d'un espace
vectoriel sur un corps fini~?
\end{question}


\begin{thebibliography}{99}

\bibitem{BH00} J. T. Baldwin, K Holland, \emph{Constructing $\omega$-stable
structures: Rank $2$ fields}, J. Symb.\ Logic, {\bf 65}, 371-391 (2000).

\bibitem{BHPW06} A. Baudisch, M. Hils, A. Martin-Pizarro, F. Wagner, \emph{Die
b\"ose Farbe}, J. Inst.\ Math.\ Jussieu, {\bf 8}, 415-443 (2009).

\bibitem{BMPZ05} A. Baudisch, A. Martin-Pizarro, M. Ziegler, \emph{Red fields},
J. Symb.\ Logic, {\bf 72}, 207--225 (2007).

\bibitem{BMPZ06} A. Baudisch, A. Martin-Pizarro, M. Ziegler, 
\emph{Fusion over a vector space}, J. Math.\ Logic {\bf 6}, 141--162 (2006).

\bibitem{BMPZ07} A. Baudisch, A. Martin-Pizarro, M. Ziegler, \emph{Hrushovski's
fusion}, dans : \emph{Algebra, logic, set theory}, 15--32, Stud.\ Log.\ 4,
Coll.\ Publ.\ London (2007).

\bibitem{BPW09} T. Blossier, A. Martin-Pizarro, F. Wagner,
\emph{Géométries relatives}, J. Europ.\ Math.\ Soc., to appear.
HAL-00514393.

\bibitem{Bou89} E. Bouscaren, \emph{The Group Configuration--after E.
{H}rushovski}, dans : \emph{The Model Theory of Groups}, 199--209, Notre Dame
Math.\ Lectures 11, University of Notre Dame Press (1989).

\bibitem{Ha08} A. Hasson, \emph{Some questions concerning Hrushovski's
amalgamation constructions}, J. Inst.\ Math.\ Jussieu, {\bf 7}, 793--823 (2008).

\bibitem{eHPhD} E. Hrushovski, \emph{Contributions to stable model theory},
Ph.D. Thesis, Berkeley (1986).

\bibitem{Hr92} E. Hrushovski, \emph{Strongly minimal expansions of algebraically
closed fields}, Isr.\ J. Math., {\bf 79}, 129--151 (1992).

\bibitem{Hr93} E. Hrushovski, \emph{A new strongly minimal set}, Ann.\ Pure
Appl.\ Logic, {\bf 62}, 147--166 (1993).

\bibitem{yM03} Y. Mustafin, Thèse de doctorat, Lyon (2003). 

\bibitem{Pi96} A. Pillay, Geometric Stability Theory, Oxford Logic Guides, 33.
\emph{Oxford University Press} (1996).


\bibitem{Po87} B. Poizat, Groupes Stables. Une tentative de conciliation entre
la g\'eom\'etrie alg\'ebrique et la logique math\'ematique, \emph{Nur al-Mantiq
wal-Ma $\Acute{}$rifah} (1987). Traduction anglaise~: Stable groups.
Mathematical Surveys and Monographs, 87. \emph{Amer.\ Math.\ Soc.} (2001).

\bibitem{Po99} B. Poizat, \emph{Le carr\'e de l'\'egalit\'e}, J. Symb. Logic,
{\bf 64}, 1339--1355 (1999).

\bibitem{Po01} B. Poizat, \emph{L'\'egalit\'e au cube}, J. Symb. Logic, {\bf
66}, 1647--1676 (2001).

\bibitem{Polong} B. Poizat, \emph{Quelques modestes remarques \`a propos d'une
cons\'equence inattendue d'un r\'esultat surprenant de {M}onsieur {F}rank {O}laf
{W}agner}, J. Symb. Logic, {\bf 66}, $\mathrm{n}^\mathrm{o}$  4, 1637--1646
(2001).

\bibitem{Wa97} F. O. Wagner, Stable groups, Lecture Notes of the London Mathematical Society 240.
Cambridge University Press, 1997.

\bibitem{Wa01} F. O. Wagner, \emph{Fields of finite {M}orley rank}, J. Symb.\
Logic, {\bf 66}, 703--706 (2001).

\bibitem{Wa03} F. O. Wagner, \emph{Bad fields in positive characteristic},
Bull.\ London Math.\ Soc., {\bf 35}, 499--502 (2003).

\bibitem{mZ90} M. Ziegler, \emph{A Note on generic Types}, 
unpublished, (2006), (\url{http://arxiv.org/math.LO/0608433}).

\bibitem{mZ08} M. Ziegler,\emph{Fusion of structures of finite {M}orley rank},
dans : \emph{Model theory with applications to algebra and analysis. {V}ol.\ 1}, 
London Math. Soc. Lecture Note Ser.,{\bf 349}, 225--248,
Cambridge Univ. Press, (2008).
\end{thebibliography}
\end{document}